\documentclass{amsart}
\usepackage{amssymb}
\usepackage{amsmath}
\usepackage{amscd}
\usepackage{multicol}
\usepackage{tikz}
\usepackage{graphics}
\usepackage{shuffle}
\usepackage[OT2,T1]{fontenc}
\usepackage[all]{xy}
\usepackage{extarrows}
\usepackage{comment}
\DeclareSymbolFont{cyrletters}{OT2}{wncyr}{m}{n}
\DeclareMathSymbol{\Sha}{\mathalpha}{cyrletters}{"58}
\DeclareMathAlphabet{\mathpzc}{OT1}{pzc}{m}{it}

\usepackage{ulem}

\title[]
{On properties of $\adari(\pal)$ and $\ganit_v(\pic)$\\
}
\author{Nao Komiyama}

\address{Graduate School of Mathematics, Nagoya University, 
Furo-cho, Chikusa-ku, Nagoya 464-8602 Japan }
\email{m15027u@math.nagoya-u.ac.jp}
\thanks{}
\date{\today}

\newtheorem{thm}{Theorem}[section]
\newtheorem{lem}[thm]{Lemma}
\newtheorem{cor}[thm]{Corollary}
\newtheorem{prop}[thm]{Proposition}  

\theoremstyle{remark}
        
\subjclass[2010]{}

\numberwithin{equation}{section}
\theoremstyle{definition}
\newtheorem{definition}[thm]{Definition}
\newtheorem{remark}[thm]{Remark}
\newtheorem{notation}[thm]{Notation}     

\newtheorem{example}[thm]{Examples}

\newcommand{\vecv}{{\bf v}}
\newcommand{\vecu}{{\bf u}}

\newcommand{\varia}[2]{({}^{#1}_{#2})}
\newcommand{\id}{{\rm id}}

\newcommand{\N}{{\mathbb N}}
\newcommand{\Z}{{\mathbb Z}}
\newcommand{\Q}{{\mathbb Q}}

\newcommand{\C}{{\mathbb C}}

\newcommand{\Sh}[3]{{\rm Sh}\binom{#1;#2}{#3}}
\newcommand{\Shstar}[3]{{\rm Sh}_*\binom{#1;#2}{#3}}

\newcommand{\shmap}{\mathpzc{Sh}}

\newcommand{\adari}{{\rm adari}}
\newcommand{\adgari}{{\rm adgari}}
\newcommand{\sh}{{\rm sh}}
\newcommand{\Bf}[1]{\mbox{\boldmath {#1}}}

\newcommand{\BIMU}{{\rm BIMU}}
\newcommand{\ARI}{{\rm ARI}}
\newcommand{\al}{{\rm al}}
\newcommand{\il}{{\rm il}}
\newcommand{\GARI}{{\rm GARI}}
\newcommand{\as}{{\rm as}}
\newcommand{\is}{{\rm is}}
\newcommand{\DIMU}{{\rm DIMU}}

\newcommand{\arit}{{\rm arit}}

\newcommand{\ari}{{\rm ari}}
\newcommand{\preari}{{\rm preari}}
\newcommand{\expari}{{\rm expari}}

\newcommand{\ganit}{{\rm ganit}}

\newcommand{\anti}{{\rm anti}}
\newcommand{\pari}{{\rm pari}}

\newcommand{\ulflex}[2]{{#1}\rceil_{#2}}

\newcommand{\lrflex}[2]{{}_{#1}\lfloor{#2}}

\newcommand{\pal}{{\rm pal}}

\newcommand{\paj}{{\rm paj}}
\newcommand{\pij}{{\rm pij}}
\newcommand{\pic}{{\rm pic}}
\newcommand{\poc}{{\rm poc}}

\begin{document}
\bibliographystyle{amsalpha+}
\maketitle

\begin{abstract}
The paper discusses properties of $\adari(\pal)$ and $\ganit_v(\pic)$ which are Ecalle's maps among certain sets of moulds related to the double shuffle relations of MZVs.
We give self-contained proof of their basic properties which are exhibited in Ecalle's papers and partially proved in Schneps' paper.
\end{abstract}

\tableofcontents
\setcounter{section}{-1}
\section{Introduction}

The \textit{multiple zeta value} (MZV for short) is a power series defined by
\begin{equation*}\label{eqn:def of MZV}
\zeta(k_1,\dots,k_r)=\sum_{0<m_1<\cdots<m_r}\frac{1}{m_1^{k_1} \cdots m_r^{k_r}},
\end{equation*}
for $k_1,\dots,k_r\in\N$ with $k_r\geq2$.
It is known that MZVs satisfy the \textit{double shuffle relations} which consist of the \textit{shuffle relations} and the \textit{harmonic relations}.
To deal with these relations systematically, the generating functions of MZVs are considered in \cite{E-ARIGARI}, \cite{E-flex}, \cite{G}, \cite{G-MPL} and \cite{IKZ}.
In these papers, the shuffle relations (resp. the harmonic relations) are lifted to the relations between the generating functions which Ecalle (\cite{E-ARIGARI}, \cite{E-flex}) calls \textit{symmetral} (resp. \textit{symmetril}) relations.\footnote{See Definition \ref{def:al,il,as,is} for detail.}
To study properties of MZVs, Ecalle introduces some groups and some Lie algebras of moulds  related with these lifted relations, and he reveals various properties of them.
Our final aim is to understand the following diagram (though which is not attained in this paper) which is displayed without proof in \cite[\S 4.7]{E-flex}:
$$
\xymatrix{
\GARI(\Gamma)_{\underline{\as}*\underline{\as}} \ar[rr]^{\adgari(\pal)} & & \GARI(\Gamma)_{\underline{\as}*\underline{\is}}  \\
\ARI(\Gamma)_{\underline{\al}*\underline{\al}} \ar[rr]_{\adari(\pal)} \ar@{->}[u]^{\expari}& & \ARI(\Gamma)_{\underline{\al}*\underline{\il}} \ar@{->}[u]_{\expari}
}
$$ 
Related to the above diagram, in \cite{SS} and \cite{S-ARIGARI}, the following property of the above map $\adari(\pal)$ is discussed.\\

\noindent
\smallskip
{\bf Theorem \ref{thm:adari(pal) induce bijection}.} (\cite[Theorem 7.2]{SS}, \cite[Theorem 4.6.1]{S-ARIGARI}) {\it The Lie algebra automorphism $\adari(\pal)$ on $\ARI(\Gamma)$ induces a bijection $\adari(\pal):\ARI(\Gamma)_{\underline{\al}*\underline{\al}} \longrightarrow \ARI(\Gamma)_{\underline{\al}*\underline{\il}}$.}\\
\smallskip

In \cite{S-ARIGARI}, the above theorem is deduced from the following.\\

\noindent
\smallskip
{\bf Theorem \ref{thm: ganit(B) is automorphism between al and il or as and is}.} (cf. \cite[Proposition 6.2]{SS}, \cite[Lemma 4.4.2]{S-ARIGARI}) {\it The map $\ganit_v(\pic)$ induces a group isomorphism from $(\overline{\GARI}(\Gamma)_\as,\times)$ to $(\overline{\GARI}(\Gamma)_\is,\times)$ and induces a Lie algebra isomorphism from $(\overline{\ARI}(\Gamma)_\al,[,])$ to $(\overline{\ARI}(\Gamma)_\il,[,])$.}\\
\smallskip

However, the proof of this statement is not fully presented.\footnote{In precise, in \cite[proof of Proposition 6.2]{SS} only the case of depth 4 is shown, and in \cite[proof of Lemma 4.4.2]{S-ARIGARI} ``A straightforward calculation'' (in \cite[the end of page 55]{S-ARIGARI}) is not straightforward at all at least for the author. We note that this issue of the calculation is justified in Remark \ref{rem: inverse map of ganit(pic)}.}
This paper gives self-contained proof of this theorem by clarifying several compatibilities of the maps $\adari(\pal)$ and $\ganit_v(\pic)$ including Theorem \ref{thm:exp(ARIal)=GARIas etc}, Corollary \ref{cor:commutative diagram} and Theorem \ref{thm: ganit(B) is automorphism between al and il or as and is}.
By using these corollary and theorems, we obtain the following commutative diagram in \cite[\S 4.7]{E-flex}. \\

\noindent
\smallskip
{\bf Corollary \ref{cor:commutative diagram of al,il,as,is}.} (\cite[\S 4.7]{E-flex}) {\it The following diagram commutes:}
$$
\xymatrix{
\overline{\GARI}(\Gamma)_\as \ar[rr]^{\ganit_v(\pic)} & & \overline{\GARI}(\Gamma)_\is  \\
\overline{\ARI}(\Gamma)_\al \ar[rr]_{\ganit_v(\pic)} \ar@{->}[u]^{\exp_\times}& & \overline{\ARI}(\Gamma)_\il \ar@{->}[u]_{\exp_\times}
}
$$
\smallskip

The construction of this paper is as follows.
In \S \ref{sec:Mould}, we review the definition (Definition \ref{def:mould}) of moulds introduced in \cite{E-ARIGARI} and \cite{E-flex} and we explain several fundamental properties (Definition \ref{def:al,il,as,is}) of moulds.
We also give some examples (Example \ref{ex:al,il,as,is mould}) of moulds, which satisfies the above properties.
In \S \ref{sec:Exponential map}, we recall the definition (Definition \ref{def:2.2.1}) of dimoulds introduced in \cite{Sau}, and by using these, we will show Theorem \ref{thm:exp(ARIal)=GARIas etc}, a correspondence result of certain sets of moulds.
In \S \ref{sec:Automorphism}, we review the definition (Definition \ref{def:ganit}) of the map $\ganit_v(B)$ introduced in \cite[(2.36)]{E-flex}, and we prove Theorem \ref{thm: ganit(B) is automorphism}, that is, $\ganit_v(B)$ forms an automorphism on $\overline{\GARI}(\Gamma)$ and $\overline{\ARI}(\Gamma)$.
In \S \ref{subsec:Algebraic preparation}, we introduce elements $g_B(\vecv_r) \in \mathcal K\langle V_\Z \rangle$ (Definition \ref{def:definition of map g}) as an analogue of $\ganit_v(B)$.
In \S \ref{subsec:map between alternal and alternil}, we prove the recurrence formula (Proposition \ref{prop:recurrence formula of map g}) of $g_B(\vecv_r)$ for $B=\pic$.
By using this recurrence formula, we prove Theorem \ref{thm: ganit(B) is automorphism between al and il or as and is}, $\ganit_v(\pic)$ induces certain isomorphisms.
As a corollary, we obtain compatibilities of $\ganit_v(\pic)$ with exponential maps (Corollary \ref{cor:commutative diagram of al,il,as,is}) and obtain a bijection between $\ARI(\Gamma)_{\underline{\al}*\underline{\al}}$ and $\ARI(\Gamma)_{\underline{\al}*\underline{\il}}$ (Theorem \ref{thm:adari(pal) induce bijection}).
As an appendix, we prove $\expari(\ARI(\Gamma)_\al)=\GARI(\Gamma)_\as$ (Theorem \ref{thm:expari(ARIal) = GARIas}), which is used to show Theorem \ref{thm:adari(pal) induce bijection} (cf. \cite[Proposition 2.6.1]{S-ARIGARI}).

\section{Moulds}\label{sec:Mould}

In \S \ref{subsec:Definition of mould}, we review the definition of moulds introduced in \cite{E-ARIGARI} and \cite{E-flex}.
For our convenience, we recall the definition by following \cite{FK}.
We introduce the alternality, the symmetrality, the alternility and the symmetrility in \S \ref{subsec:al,il,as,is}.
In \S \ref{subsec:al,il,as,is}, we give several examples of moulds, some of which will play an important role in later sections.

\subsection{Definition of moulds}\label{subsec:Definition of mould}

Let $\Gamma$ be a finite abelian group.
We set\footnote{We may take $\mathcal F$ as the field of all meromorphic functions or the Laurent series ring.}
$\mathcal{F}:=\bigcup_{m\geqslant 1}\Q (x_1,\dots,x_m)$.

\begin{definition}[{\cite[Definition 1.1]{FK}}]\label{def:mould}
A {\it mould} on $\Z_{\geqslant0}$ with values in $\mathcal{F}$ 
is  a collection (a sequence) 
\begin{equation*}
	M=(M(x_1,\dots, x_m))_{m\in\Z_{\geqslant0}}=\bigl(M(\emptyset),\ M(x_1),\ M(x_1, x_2),\ \dots\bigr),
\end{equation*}
with $M(\emptyset)\in\mathbb Q$ and $M(x_1,\dots,x_m)$
$\in\mathbb Q(x_1,\dots,x_m)^{\oplus\Gamma^{\oplus m}}$ for $m\geqslant 1$,
which is described by a summation
$$
M(x_1,\dots, x_m)=
\underset{(\sigma_1,\dots,\sigma_m)\in\Gamma^{\oplus m}}{\oplus}
M_{\sigma_1, \dots, \sigma_m}(x_1,\ \dots,\ x_m)
$$
where each $M_{\sigma_1, \dots, \sigma_m}(x_1,\ \dots,\ x_m)\in\mathbb Q(x_1,\dots,x_m)$. 
We denote the set of all moulds with values
in $\mathcal{F}$ by $\mathcal{M}(\mathcal{F};\Gamma)$. 
The set $\mathcal M(\mathcal F;\Gamma)$ forms a $\Q$-linear space by
\begin{align*}
	A+ B
	&:= (A(x_1,\dots, x_m)+ B(x_1,\dots, x_m))_{m\in\Z_{\geqslant0}}, \\
	c A
	&:= (c A(x_1,\dots, x_m))_{m\in\Z_{\geqslant0}},
\end{align*}
for $A, B\in\mathcal M(\mathcal F;\Gamma)$ and $c\in \Q$,
namely the addition and the scalar are taken componentwise.
We define a product on $\mathcal M(\mathcal F;\Gamma)$ by
\begin{equation*}
	 (A\times B)_{\sigma_1, \dots, \sigma_m}(x_1, \dots, x_m)
	:=\sum_{i=0}^mA_{\sigma_1, \dots, \sigma_i}(x_1, \dots, x_{i})
	B_{\sigma_{i+1}, \dots, \sigma_m}(x_{i+1}, \dots, x_m),
\end{equation*}
for $A,B\in\mathcal M(\mathcal F;\Gamma)$ and for $m\geqslant0$ and for $(\sigma_1,\dots,\sigma_m)\in\Gamma^{\oplus m}$. 
Then the pair $(\mathcal M(\mathcal F;\Gamma),\times)$ is a non-commutative, associative, unital $\Q$-algebra. Here, the unit $ I\in\mathcal M(\mathcal F;\Gamma)$ is given by 
$ I:=(1,0,0,\dots)$.
\end{definition}

We prepare $\BIMU(\Gamma)$ and $\overline{\BIMU}(\Gamma)$ as copies of $\mathcal M(\mathcal F;\Gamma)$.
\begin{remark}\label{rmk:biomould}
We often regard our moulds in Definition \ref{def:mould} as if a bimould introduced in \cite{E-ARIGARI} and \cite{E-flex}, and we sometimes denote
$M_{\sigma_1, \dots, \sigma_m}(x_1,\ \dots,\ x_m)$ by $M\varia{x_1,\ \dots,\ x_m}{\sigma_1,\ \dots,\ \sigma_m}$ for $M\in\BIMU(\Gamma)$ and denote
$M\varia{\sigma_1,\ \dots,\ \sigma_m}{x_1,\ \dots,\ x_m}$ for $M\in\overline{\BIMU}(\Gamma)$.
Moreover, we sometimes use $u_i$ instead of $x_i$ for $\BIMU(\Gamma)$ and use $v_i$ instead of $x_i$ for $\overline{\BIMU}(\Gamma)$.
\end{remark}

\begin{notation}\label{not:on expression of elements in V_Z}
We often use the following algebraic formulation which is useful to describe several properties of moulds:
Put $U:=\left\{\binom{u_i}{\sigma}\right\}_{i\in\N,\sigma\in\Gamma}$ and $V:=\left\{\binom{\sigma}{v_i}\right\}_{i\in\N,\sigma\in\Gamma}$. Let $U_\Z$ and $V_\Z$ be the sets such that
\begin{align*}
	&U_\Z:=\left\{\varia{u}{\sigma}\ \middle|\ u=a_1u_1+\cdots +a_ku_k,\ k\in\N,\ a_j\in\Z,\ \sigma\in\Gamma\right\}, \\
	&V_\Z:=\left\{\varia{\sigma}{v}\ \middle|\ v=a_1v_1+\cdots +a_kv_k,\ k\in\N,\ a_j\in\Z,\ \sigma\in\Gamma\right\}.
\end{align*}
and let
$U_\Z^\bullet$ (resp. $V_\Z^\bullet$)
be the non-commutative free monoid generated by all elements of $U_\Z$ (resp. $V_\Z$)
with the empty word $\emptyset$ as the unit. 
Occasionally we denote each element $\omega=u_1\cdots u_m\in U_\Z^\bullet$
with $u_1,\dots,u_m\in U_\Z$ by
$\omega=(u_1,\dots,u_m)$
as a sequence.
The {\it length} of $\omega=(u_1,\dots,u_m)$ is defined to be $l(\omega):=m$.
We define the length $l(\omega)$ for any elements $\omega\in V_\Z^\bullet$ in the same way.
\end{notation}

For our simplicity we  occasionally denote $M\in\BIMU(\Gamma)$ (resp. $\in\overline{\BIMU}(\Gamma)$) by
\begin{align}\label{eqn:notation of bimould}
	M=( M(\vecu_m))_{m\in\Z_{\geqslant0}},\qquad
	(\mbox{resp. } M=( M(\vecv_m))_{m\in\Z_{\geqslant0}})
\end{align}
where $\vecu_0:=\emptyset$ and $\vecu_m:=\varia{u_1,\ \dots,\ u_m}{\sigma_1,\ \dots,\ \sigma_m}$ (resp. $\vecv_0:=\emptyset$ and $\vecv_m:=\varia{\sigma_1,\ \dots,\ \sigma_m}{v_1,\ \dots,\ v_m}$) for $m\geqslant1$. Under the notations,  the product of $ A, B\in\BIMU(\Gamma)$ (resp. $\in\overline{\BIMU}(\Gamma)$) is expressed as
\begin{equation*}
	 A\times  B
	=\left(\sum_{\substack{
		\vecu_m=\alpha\beta}}
	A(\alpha)B(\beta)\right)_{m\in\Z_{\geqslant0}}
	(\mbox{resp. }=\left(\sum_{\substack{
		\vecv_m=\alpha\beta}}
	A(\alpha)B(\beta)\right)_{m\in\Z_{\geqslant0}})
\end{equation*}
where $\alpha$ and $\beta$ run over $U_{\Z}^\bullet$ (resp. $V_{\Z}^\bullet$).

We put
\begin{align*}
	\ARI(\Gamma)&:=\{ M\in\BIMU(\Gamma)\ |\  M(\emptyset)=0\}, \\
	\GARI(\Gamma)&:=\{ M\in\BIMU(\Gamma)\ |\  M(\emptyset)=1\}.
\end{align*}
By replacing $\BIMU(\Gamma)$ to $\overline{\BIMU}(\Gamma)$, two sets $\overline{\ARI}(\Gamma)$ and $\overline{\GARI}(\Gamma)$ are also defined.
We set 
\begin{equation}\label{eqn:bracket product}
	[A,B]:=A\times B - B\times A,
\end{equation}
for $A,B\in\mathcal M(\mathcal F;\Gamma)$.
Then we see that $(\ARI(\Gamma), [,])$ and $(\overline{\ARI}(\Gamma), [,])$ form Lie algebras and $(\GARI(\Gamma),\times)$ and $(\overline{\GARI}(\Gamma),\times)$ form groups.

\subsection{Alternality and symmetrality, alternility and symmetrility}\label{subsec:al,il,as,is}

We put
$\mathcal A_U:=\Q \langle U_\Z \rangle$
to be the non-commutative polynomial $\Q$-algebra generated by 
$U_\Z$
(i.e. $\mathcal A_U$ is the $\Q$-linear space generated by $U_\Z^\bullet$).
We equip $\mathcal A_U$ a product $\shuffle:\mathcal A_U^{\otimes2} \rightarrow \mathcal A_U$ 
which is linearly defined by $\emptyset\, \shuffle\, \omega:=\omega\, \shuffle\, \emptyset:=w$ and
\begin{equation}\label{eqn:shuffle product}
	a\omega\ \shuffle\ b\eta
	:=a(\omega\, \shuffle\, b\eta)+b(a\omega\, \shuffle\, \eta),
\end{equation}
for $a,b\in U_\Z$
and $\omega,\eta\in U_\Z^\bullet$.
Then the pair $(\mathcal A_U,\shuffle)$ forms a commutative, associative, unital $\Q$-algebra.\footnote{See \cite{Re} for detail.}
Let $\left\{ \Sh{\omega}{\eta}{\alpha} \right\}_{\omega,\eta,\alpha\in U_\Z^\bullet}$ to be the family in $\Z$ defined by
\begin{equation}\label{eqn:shuffle product, coefficient def}
\omega \shuffle \eta
=\sum_{\alpha\in U_\Z^\bullet} \Sh{\omega}{\eta}{\alpha}\alpha.
\end{equation}

We put $\mathcal K:=\Q(v_i\ |\ i\in\N)$, that is, the commutative field generated by all $v_i$ over $\Q$.
We define
$\mathcal A_V:=\mathcal K \langle V_\Z \rangle$
to be the non-commutative polynomial $\mathcal K$-algebra generated by 
$V_\Z$
(i.e. $\mathcal A_V$ is the $\mathcal K$-linear space generated by $V_\Z^\bullet$).
As with $\mathcal A_U$, the algebra $\mathcal A_V$ is equipped the product $\shuffle$ and the pair $(\mathcal A_V,\shuffle)$ forms a commutative, associative, unital $\mathcal K$-algebra.
While, we also equip $\mathcal A_V$ a product $\shuffle_*:\mathcal A_V^{\otimes2} \rightarrow \mathcal A_V$ 
which is linearly defined by $\emptyset\, \shuffle_*\, \omega:=\omega\, \shuffle_*\, \emptyset:=w$ and $\binom{\sigma}{v}\omega\, \shuffle_*\, \binom{\sigma'}{v'}\eta:=0$ for $\binom{\sigma}{v},\binom{\sigma'}{v'}\in V_\Z$ with $v= v'$
and $\omega,\eta\in V_\Z^\bullet$, and
\begin{align}\label{eqn:shufflestar product}
	\varia{\sigma}{v}\omega\, \shuffle_*\, \varia{\sigma'}{v'}\eta
	&:=\varia{\sigma}{v}\Bigl( \omega\, \shuffle_*\, \varia{\sigma'}{v'}\eta \Bigr)
		+\varia{\sigma'}{v'}\Bigl( \varia{\sigma}{v}\omega\, \shuffle_*\, \eta \Bigr) \\
	&\hspace{1.5cm}+\frac{1}{v-v'}\left\{\varia{\sigma\sigma'}{\ v}(\omega\, \shuffle_*\, \eta)
	-\varia{\sigma\sigma'}{\ v'}(\omega\, \shuffle_*\, \eta)\right\} \nonumber
\end{align}
for $\binom{\sigma}{v},\binom{\sigma'}{v'}\in V_\Z$ with $v\neq v'$
and $\omega,\eta\in V_\Z^\bullet$.
Then the pair $(\mathcal A_V,\shuffle_*)$ forms a commutative, non-associative, unital $\mathcal K$-algebra.
Let $\left\{ \Shstar{\omega}{\eta}{\alpha} \right\}_{\omega,\eta,\alpha\in V_\Z^\bullet}$ to be the family in $\mathcal K$ defined by
$$
\omega \shuffle_* \eta
=\sum_{\alpha\in V_\Z^\bullet} \Shstar{\omega}{\eta}{\alpha}\alpha.
$$

For $\omega=(u_1,\cdots, u_m)\in V_\Z^\bullet$ with $u_1,\dots,u_m\in V_\Z$, we denote $L(\omega):=\{u_1,\dots,u_m\}$ to be the set of all letters appearing in $\omega$.
Then for $\omega,\eta\in V_\Z^\bullet$, two words $\omega$ and $\eta$ have no same letters if and only if $L(\omega)\cap L(\eta)=\emptyset$ holds.
\begin{lem}[{cf. \cite[Lemma A.7]{FK}}]\label{lem:shuffle coefficient}
Let $r\geq2$. For $\omega,\eta,\alpha_1,\dots,\alpha_r\in U_{\Z}^\bullet$, we have
\begin{equation}\label{eqn:shuffle coefficient}
	\Sh{\omega}{\eta}{\alpha_1\cdots\alpha_r}
	=\sum_{\substack{
	\omega=\omega_1\cdots\omega_r \\	
	\eta=\eta_1\cdots\eta_r	}}
	\Sh{\omega_1}{\eta_1}{\alpha_1}\cdots\Sh{\omega_r}{\eta_r}{\alpha_r},
\end{equation}
and for $\omega,\eta\in V_{\Z}^\bullet$ with $L(\omega)\cap L(\eta)=\emptyset$ and for $\alpha_1,\dots,\alpha_r\in V_{\Z}^\bullet$, we have
\begin{equation}\label{eqn:shufflestar coefficient}
	\Shstar{\omega}{\eta}{\alpha_1\cdots\alpha_r}
	=\sum_{\substack{
	\omega=\omega_1\cdots\omega_r \\	
	\eta=\eta_1\cdots\eta_r	}}
	\Shstar{\omega_1}{\eta_1}{\alpha_1}\cdots\Shstar{\omega_r}{\eta_r}{\alpha_r},
\end{equation}
where $\omega_1,\dots,\omega_r,\eta_1,\dots,\eta_r$ run over $U_\Z^\bullet$.
\end{lem}
\begin{proof}
The equation \eqref{eqn:shuffle coefficient} is proved in \cite[Lemma A.7]{FK}.
By the same way, we obtain \eqref{eqn:shufflestar coefficient}.
\end{proof}

\begin{definition}[{cf. \cite[Definition 1.4]{FK}}]\label{def:al,il,as,is}
A mould $M\in \ARI(\Gamma)$ (resp. $\in\GARI(\Gamma)$) is called {\it alternal} (resp. {\it symmetral}) if we have
\begin{align}\label{eqn:def of al,as}
	&\sum_{\alpha\in U_\Z^\bullet}
	\Sh{\varia{u_1,\ \dots,\ u_p}{\sigma_1,\ \dots,\ \sigma_p}}{\varia{u_{p+1},\ \dots,\ u_{p+q}}	{\sigma_{p+1},\ \dots,\ \sigma_{p+q}}}{\alpha} M(\alpha)=0 \\
	&\hspace{5cm}
	(\mbox{resp. } =M\varia{u_1,\ \dots,\ u_p}{\sigma_1,\ \dots,\ \sigma_p}M\varia{u_{p+1},\ \dots,\ u_{p+q}}{\sigma_{p+1},\ \dots,\ \sigma_{p+q}}) \nonumber
\end{align}
for $p,q\geq1$.
We denote $\ARI(\Gamma)_\al$ (resp. $\GARI(\Gamma)_\as$) to be the subset consisting of alternal (resp. symmetral) moulds.
While, a mould $M\in \overline{\ARI}(\Gamma)$ (resp. $\in\overline{\GARI}(\Gamma)$) is called {\it alternil} (resp. {\it symmetril}) if we have
\begin{align}\label{eqn:def of il,is}
	&\sum_{\alpha\in V_\Z^\bullet}
	\Shstar{\varia{\sigma_1,\ \dots,\ \sigma_p}{v_1,\ \dots,\ v_p}}{\varia{\sigma_{p+1},\ \dots,\ \sigma_{p+q}}{v_{p+1},\ \dots,\ v_{p+q}}}{\alpha} M(\alpha)=0 \\
	&\hspace{5cm}
	(\mbox{resp. } =M\varia{\sigma_1,\ \dots,\ \sigma_p}{v_1,\ \dots,\ v_p}M\varia{\sigma_{p+1},\ \dots,\ \sigma_{p+q}}{v_{p+1},\ \dots,\ v_{p+q}}) \nonumber
\end{align}
for $p,q\geq1$.
\end{definition}
\begin{remark}
For any mould $M\in \overline{\ARI}(\Gamma)$ (resp. $\in\overline{\GARI}(\Gamma)$), the alternality (resp. symmetrality) is also defined, that is, $M$ is called {\it alternal} (resp. {\it symmetral}) if we have
\begin{align}\label{eqn:def of al,as on overline(ARI)}
	&\sum_{\alpha\in V_\Z^\bullet}
	\Sh{\varia{\sigma_1,\ \dots,\ \sigma_p}{v_1,\ \dots,\ v_p}}{\varia	{\sigma_{p+1},\ \dots,\ \sigma_{p+q}}{v_{p+1},\ \dots,\ v_{p+q}}}{\alpha} M(\alpha)=0 \\
	&\hspace{5cm}
	(\mbox{resp. } =M\varia{\sigma_1,\ \dots,\ \sigma_p}{v_1,\ \dots,\ v_p}M\varia{\sigma_{p+1},\ \dots,\ \sigma_{p+q}}{v_{p+1},\ \dots,\ v_{p+q}}) \nonumber
\end{align}
for $p,q\geq1$.
\end{remark}

We denote $\overline{\ARI}(\Gamma)_\al$ (resp. $\overline{\ARI}(\Gamma)_\il$) to be the subset of $\overline{\ARI}(\Gamma)$ consisting of alternal (resp. alternil) moulds, and we denote $\overline{\GARI}(\Gamma)_\as$ (resp. $\overline{\GARI}(\Gamma)_\is$) to be the subset of $\overline{\GARI}(\Gamma)$ consisting of symmetral (resp. symmetril) moulds.

Our purpose in this paper is to give self-contained proof of various properties of these sets (see \S \ref{sec:Exponential map} and \S \ref{sec:Automorphism}).

\begin{example}\label{ex:al,il,as,is mould}
We give examples of Definition \ref{def:al,il,as,is}.
\begin{enumerate}
\renewcommand{\labelenumi}{(\alph{enumi}).}
	\item The mould $A\in \ARI(\Gamma)$ is defined by $A(\vecu_0):=A(\vecu_1):=0$ and for $m\geq2$
	\begin{equation*}
		A(\vecu_m):=\frac{1}{(u_2-u_1)\cdots(u_m-u_{m-1})}.
	\end{equation*}
	Then this mould $A$ is alternal.
	\item The mould $\paj\in \GARI(\Gamma)$ is defined by $\paj(\vecu_0):=1$ and for $m\geq1$
	\begin{equation*}
		\paj(\vecu_m):=\frac{1}{u_1(u_1+u_2)\cdots(u_1+\cdots+u_m)}.
	\end{equation*}
	Then this mould $\paj$ is symmetral.
	\item The mould $C\in \overline{\ARI}(\Gamma)$ is defined by $C(\vecv_0):=0$ and $C(\vecv_1):=1$ and for $m\geq2$
	\begin{equation*}
		C(\vecv_m):=\frac{1}{(v_2-v_1) \cdots (v_m-v_1)}.
	\end{equation*}
	Then this mould $C$ is alternil.
	\item The mould $\pic\in \overline{\GARI}(\Gamma)$ is defined by $\pic(\vecv_0):=1$ and for $m\geq1$
	\begin{equation*}
		\pic(\vecv_m):=\frac{1}{v_1v_2 \cdots v_m}.
	\end{equation*}
	Then this mould $\pic$ is symmetril.
	\item The mould $\pij\in \overline{\GARI}(\Gamma)$ is defined by $\pij(\vecv_0):=1$ and $\pij(\vecv_1):=\frac{1}{v_1}$ and for $m\geq2$
	\begin{equation*}
		\pij(\vecv_m):=\frac{1}{(v_1-v_2)(v_2-v_3) \cdots (v_{m-1}-v_m)v_m}.
	\end{equation*}
	Then this mould $\pij$ is symmetral.
\end{enumerate}
\end{example}

\begin{remark}\label{rem:shufflestar product-pic ver.}
We mention two properties of the mould $\pic\in \overline{\GARI}(\Gamma)$:
\begin{enumerate}
\renewcommand{\labelenumi}{(\alph{enumi}).}
\item We put $\omega_1=\varia{\sigma}{v}$ and $\eta_1=\varia{\sigma'}{v'}$.
By using the mould $\pic$ and using flexions in \cite[Definition 1.8]{FK}, we can rewrite \eqref{eqn:shufflestar product} as
\begin{align}\label{eqn:shufflestar product-pic ver.}
	(\omega_1,\omega)\, \shuffle_*\, (\eta_1,\eta)
	&:=(\omega_1)\left( \omega\, \shuffle_*\, (\eta_1,\eta) \right)
		+(\eta_1)\left((\omega_1,\omega)\, \shuffle_*\, \eta \right) \\
	&\quad-\pic(\lrflex{\omega_1}{\eta_1})\ (\ulflex{\omega_1}{\eta_1},\,\omega\, \shuffle_*\, \eta)
	-\pic(\lrflex{\eta_1}{\omega_1})\ (\ulflex{\eta_1}{\omega_1},\,\omega\, \shuffle_*\, \eta). \nonumber
\end{align}
\item By the definition of the mould $\pic$, the following relations hold:
\begin{align}
\label{eqn:pic,upperflex}& \pic(\ulflex{\omega_1}{\omega_2})=\pic(\omega_1), \\
\label{eqn:pic,lowerflex}& \pic(\lrflex{\omega_1}{\omega_2})=-\pic(\lrflex{\omega_2}{\omega_1}), \\
\label{eqn:pic,Fay identity}& \pic(\lrflex{\omega_1}{(\omega_2,\omega_3)})+\pic(\lrflex{\omega_2}{(\omega_3,\omega_1)})+\pic(\lrflex{\omega_3}{(\omega_1,\omega_2)})=0, \\
\label{eqn:pic,decomposition}& \pic(\omega)\pic(\eta)=\pic(\omega,\eta),
\end{align}
for $\omega_1,\omega_2,\omega_3\in V_\Z$ and $\omega,\eta\in V_\Z^\bullet$.
\end{enumerate}
We often use these properties to prove the alternility and the symmetrility.
\end{remark}

\section{Exponential maps}\label{sec:Exponential map}

In this section, we recall the definition of dimoulds introduced in \cite{Sau}, and by using dimoulds, we prove theorem on the exponential map.
In \S \ref{subsec:Dimould}, we review the definition of dimoulds, and we show equivalent conditions of the alternality, the symmetrality and the alternility, the symmetrility for moulds.
In \S \ref{subsec:Definition of exp}, by using these equivalent conditions, we prove Theorem \ref{thm:exp(ARIal)=GARIas etc}, which is not presented in \cite{E-flex}, but is required in \S \ref{sec:Automorphism}.

\subsection{Dimoulds}\label{subsec:Dimould}
\begin{definition}[cf. {\cite[Definition 5.2]{Sau}}]\label{def:2.2.1}
A {\it dimould}
\footnote{We note that the definition of the dimould is different from the definition of the bimould introduced in \cite{E-ARIGARI} and \cite{E-flex}.}
$M^{\bullet,\bullet}$ with values in $\mathcal F$ is a sequence:
\begin{equation*}
	M^{\bullet,\bullet}:=(M^{\bullet,\bullet}(x_1,\dots,x_r ;\ x_{r+1},\dots,x_{r+s}))_{r,s\in\mathbb{N}_0}.
\end{equation*}
with $M^{\bullet,\bullet}(\emptyset;\emptyset) \in \Q$ and $M^{\bullet,\bullet}(x_1,\dots,x_r ;\ x_{r+1},\dots,x_{r+s}) \in \mathbb Q(x_1,\dots,x_{r+s})^{\oplus\Gamma^{\oplus r+s}}$ for $r\geq1$ or $s\geq1$, which is described by a summation
\begin{align*}
	&M^{\bullet,\bullet}(x_1,\dots,x_r ;\ x_{r+1},\dots,x_{r+s}) \\
	&=\underset{(\sigma_1,\dots,\sigma_{r+s})\in\Gamma^{\oplus {r+s}}}{\oplus}
	M^{\bullet,\bullet}_{\sigma_1, \dots, \sigma_r;\sigma_{r+1}, \dots, \sigma_{r+s}}(x_1,\dots,x_r ;\ x_{r+1},\dots,x_{r+s})
\end{align*}
where each $M^{\bullet,\bullet}_{\sigma_1, \dots, \sigma_r;\sigma_{r+1}, \dots, \sigma_{r+s}}(x_1,\dots,x_r ;\ x_{r+1},\dots,x_{r+s}) \in \mathbb Q(x_1,\dots,x_{r+s})$. 
We denote the set of all dimoulds with values in $\mathcal F$ by $\mathcal{M}_2(\mathcal{F};\Gamma)$.
By the component-wise summation and the component-wise scalar multiple, the set $\mathcal{M}_2(\mathcal{F};\Gamma)$ forms a $\Q$-linear space.
The product of $\mathcal{M}_2(\mathcal{F};\Gamma)$ which is denoted by the same symbol $\times$ as the product of $\mathcal{M}(\mathcal{F};\Gamma)$ is defined by
\begin{align}\label{eqn:2.2.1}
	&(A^{\bullet,\bullet}\times B^{\bullet,\bullet})_{\sigma_1, \dots, \sigma_r;\sigma_{r+1}, \dots, \sigma_{r+s}}(x_1,\dots,x_r ;\ x_{r+1},\dots,x_{r+s}) \\
	&:=\sum_{i=0}^r\sum_{j=0}^s
	A^{\bullet,\bullet}_{\sigma_1, \dots, \sigma_i;\sigma_{r+1}, \dots, \sigma_{r+j}}(x_1,\dots,x_i ;\ x_{r+1},\dots,x_{r+j}) \nonumber \\
	&\hspace{2cm} \cdot B^{\bullet,\bullet}_{\sigma_{i+1}, \dots, \sigma_r;\sigma_{r+j+1}, \dots, \sigma_{r+s}}(x_{i+1},\dots,x_r\ ;\ x_{r+j+1},\dots,x_{r+s}), \nonumber
\end{align}
for $A^{\bullet,\bullet},B^{\bullet,\bullet}\in\mathcal{M}_2(\mathcal{F};\Gamma)$ and for $r,s\geqslant0$ and for $(\sigma_1,\dots,\sigma_{r+s}) \in \Gamma^{\oplus r+s}$.
Then $(\mathcal{M}_2(\mathcal{F};\Gamma), \times)$ is non-commutative, associative $\Q$-algebra.
Here, the unit $I^{\bullet,\bullet}$ of $(\mathcal{M}_2(\mathcal{F};\Gamma), \times)$ is defined as follows:
\begin{equation*}
	I^{\bullet,\bullet}(\omega;\eta):=\left\{
	\begin{array}{ll}
	1 & (\omega=\eta=\emptyset), \\
	0 & (\mbox{otherwise}).
	\end{array}\right.
\end{equation*}
\end{definition}

We prepare $\DIMU(\Gamma)$ and $\overline{\DIMU}(\Gamma)$ as copies of $\mathcal{M}_2(\mathcal{F};\Gamma)$.
\begin{remark}
Similar to the notation of Remark \ref{rmk:biomould}, we sometimes denote \\
$M^{\bullet,\bullet}_{\sigma_1, \dots, \sigma_r;\sigma_{r+1}, \dots, \sigma_{r+s}}(x_1,\dots,x_r ;\ x_{r+1},\dots,x_{r+s})$ by
$$
M^{\bullet,\bullet}\varia{u_1,\ \dots,\ u_r;\ u_{r+1},\ \dots,\ u_{r+s}}{\sigma_1,\ \dots,\ \sigma_r;\ \sigma_{r+1},\ \dots,\ \sigma_{r+s}}
\quad (\mbox{resp.  } M^{\bullet,\bullet}\varia{\sigma_1,\ \dots,\ \sigma_r;\ \sigma_{r+1},\ \dots,\ \sigma_{r+s}}{v_1,\ \dots,\ v_r\,;\ v_{r+1},\ \dots,\ v_{r+s}})
$$
for $M^{\bullet,\bullet}\in \DIMU(\Gamma)$ (resp. $\in \overline{\DIMU}(\Gamma)$).
As with the notation \eqref{eqn:notation of bimould}, we often denote $M^{\bullet,\bullet}\in\DIMU(\Gamma)$ (resp. $\in\overline{\DIMU}(\Gamma)$) as
\begin{align*}
	M^{\bullet,\bullet}=( M^{\bullet,\bullet}(\omega; \eta))_{\omega, \eta}
\end{align*}
by putting $\omega:=\varia{u_1,\ \dots,\ u_r}{\sigma_1,\ \dots,\ \sigma_r}$ and $\eta:=\varia{u_{r+1},\ \dots,\ u_{r+s}}{\sigma_{r+1},\ \dots,\ \sigma_{r+s}}$ (resp. $\omega:=\varia{\sigma_1,\ \dots,\ \sigma_r}{v_1,\ \dots,\ v_r}$ and $\eta:=\varia{\sigma_{r+1},\ \dots,\ \sigma_{r+s}}{v_{r+1},\ \dots,\ v_{r+s}}$).
By using this notation, we can rewrite \eqref{eqn:2.2.1} by
\begin{align*}
	A^{\bullet,\bullet}\times B^{\bullet,\bullet}
	=\left(A^{\bullet,\bullet}\times B^{\bullet,\bullet}(\omega;\eta)\right)_{\omega,\eta}
	=\left(\sum_{\substack{
		\omega=\omega_1\omega_2 \\
		\eta=\eta_1\eta_2}}
	A^{\bullet,\bullet}(\omega_1;\eta_1)B^{\bullet,\bullet}(\omega_2;\eta_2)\right)_{\omega,\eta},
\end{align*}
where $\omega_1,\omega_2,\eta_1,\eta_2$ run over $U_\Z^\bullet$ (resp. $V_\Z^\bullet$).
\end{remark}
We introduce two maps from the set $\overline{\BIMU}(\Gamma)$ of moulds to the set $\overline{\DIMU}(\Gamma)$ of dimoulds.
\begin{definition}[{\cite[\S 5.2]{Sau}}]
(i): The $\Q$-linear map
$\shmap: \overline{\BIMU}(\Gamma) \rightarrow\overline{\DIMU}(\Gamma)$ is defined by
\begin{align*}
	\shmap(M)
	&=\left(\shmap(M)(\omega;\eta)\right)_{\omega,\eta}
	:=\left(\sum_{\alpha\in V_\Z^\bullet}\Sh{\omega}{\eta}{\alpha}M(\alpha)\right)_{\omega,\eta}
\end{align*}
for $M\in\overline{\BIMU}(\Gamma)$.\\
(ii): The $\Q$-linear map $i:\overline{\BIMU}(\Gamma)\otimes_\Q \overline{\BIMU}(\Gamma)\rightarrow\overline{\DIMU}(\Gamma)$ is defined by
\begin{align*}
	i(M\otimes N)
	&=\left(i(M\otimes N)(\omega;\eta)\right)_{\omega,\eta} \\
	&:=\left(M(\omega)N(\eta)\right)_{\omega,\eta}
\end{align*}
for $M,N\in\overline{\BIMU}(\Gamma)$.
It is clear that $i$ is injective. For our simplicity, we denote $i(M\otimes N)$ by $M\otimes N$.
\end{definition}
By using these two maps, we reformulate the symmetral moulds and the alternal moulds in terms of dimoulds.
\begin{prop}[{\cite[Lemma 5.2.]{Sau}}]\label{prop:gp-like, Lie-like}
For mould $M\in\overline{\BIMU}(\Gamma)$, we have
\begin{align*}
	{\rm (i)}:&M\in\overline{\ARI}(\Gamma)_\al\Longleftrightarrow
	\shmap(M)=M\otimes I+I\otimes M, \\
	{\rm (ii)}:&M\in\overline{\GARI}(\Gamma)_\as\Longleftrightarrow
	\shmap(M)=M\otimes M.
\end{align*}
\end{prop}
\begin{proof}
Because we can prove (i) as with (ii), we only prove (i). \\
($\Leftarrow$): Let $M\in\overline{\BIMU}(\Gamma)$. By the definition of the map $\mathpzc{Sh}$, we have
\begin{equation*}
	\shmap(M)(\omega;\eta)
	=\left\{\begin{array}{ll}
		M(\emptyset) & (\omega=\eta=\emptyset), \\
		M(\eta) & (\omega=\emptyset,\ \eta\neq\emptyset), \\
		M(\omega) & (\omega\neq\emptyset,\ \eta=\emptyset), \\
		\displaystyle\sum_{\alpha\in V_\Z^\bullet}\Sh{\omega}{\eta}{\alpha}M(\alpha) & (\omega,\eta\neq\emptyset).
	\end{array}\right.
\end{equation*}
While, we have
\begin{equation*}
	(M\otimes I+I\otimes M)(\omega;\eta)
	=M(\omega)I(\eta)+I(\omega)M(\eta)
	=\left\{\begin{array}{ll}
		2\cdot M(\emptyset) & (\omega=\eta=\emptyset), \\
		M(\eta) & (\omega=\emptyset,\ \eta\neq\emptyset), \\
		M(\omega) & (\omega\neq\emptyset,\ \eta=\emptyset), \\
		0 & (\omega,\eta\neq\emptyset).
	\end{array}\right.
\end{equation*}
If we have $\shmap(M)=M\otimes I+I\otimes M$, then we get $M(\emptyset)=0$ and
\begin{align*}
	\sum_{\alpha\in V_\Z^\bullet}\Sh{\omega}{\eta}{\alpha}M(\alpha)=0,
\end{align*}
for $\omega,\eta\neq\emptyset$.
Therefore, we obtain $M\in\overline{\ARI}(\Gamma)_\al$.
By observing this discussion, we get the proof of $(\Rightarrow)$.
\end{proof}

\begin{cor}\label{cor:2.3.1}
We have
\begin{equation*}
	\mathpzc{Sh}(I)=I \otimes I = I^{\bullet,\bullet}.
\end{equation*}
\end{cor}
\begin{proof}
The first equality is from $I\in\overline{\GARI}(\Gamma)$ and Proposition \ref{prop:gp-like, Lie-like}, the second equality is from the definition of the both side.
\end{proof}
We prove two important lemmas.
\begin{lem}[{\cite[Lemma 5.1]{Sau}}]\label{lem:shuffle map is alg. hom.}
The map $\mathpzc{Sh}$ is an algebra homomorphism, that is, we have
\begin{equation}\label{eqn:shuffle map is alg. hom.}
	\shmap(M\times N)=\shmap(M)\times \shmap(N),
\end{equation}
for $M,N\in\overline{\BIMU}(\Gamma)$.
\end{lem}
\begin{proof}
By using the definition of the map $\shmap$ and the product $\times$, we have
\begin{align*}
	\shmap(M\times N) (\omega;\eta)
	&=\sum_{\alpha\in V_\Z^\bullet}\Sh{\omega}{\eta}{\alpha}
	\sum_{\alpha=\alpha_1\alpha_2} M(\alpha_1)N(\alpha_2) \\
	&=\sum_{\alpha_1,\alpha_2\in V_\Z^\bullet}\Sh{\omega}{\eta}{\alpha_1\alpha_2}
	M(\alpha_1)N(\alpha_2).
\intertext{By using \eqref{eqn:shuffle coefficient} for $r=2$, we get}
	&=\sum_{\alpha_1,\alpha_2\in V_\Z^\bullet}
	\sum_{\substack{
	\omega=\omega_1\omega_2 \\	
	\eta=\eta_1\eta_2	}}
	\Sh{\omega_1}{\eta_1}{\alpha_1}\Sh{\omega_2}{\eta_2}{\alpha_2}
	M(\alpha_1)N(\alpha_2) \\
	&=\sum_{\substack{
	\omega=\omega_1\omega_2 \\	
	\eta=\eta_1\eta_2	}} \shmap(M)(\omega_1;\eta_1)\shmap(N)(\omega_2;\eta_2) \\
	&=(\shmap(M)\times \shmap(N))(\omega;\eta).
\end{align*}
Hence, we finish the proof.
\end{proof}
\begin{remark}\label{rem:gp-like, Lie-like of il, is}
We define the $\Q$-linear map $\shmap_*:\overline{\BIMU}(\Gamma) \rightarrow\overline{\DIMU}(\Gamma)$ by
\begin{align}\label{eqn:shuffle* map}
	\shmap_*(M)
	=\left(\mathpzc{Sh}_*(M)(\omega;\eta)\right)_{\omega,\eta}
	:=\left(\sum_{\alpha\in V_\Z^\bullet}\Shstar{\omega}{\eta}{\alpha}M(\alpha)\right)_{\omega,\eta},
\end{align}
for $M\in\overline{\BIMU}(\Gamma)$.
Then we get the following two:
\begin{align*}
	{\rm (i).\ } &M\in\overline{\ARI}(\Gamma)_\il\Longleftrightarrow
	\mathpzc{Sh}_*(M)=M\otimes I+I\otimes M, \\
	{\rm (ii).\ } &M\in\overline{\GARI}(\Gamma)_\is\Longleftrightarrow
	\mathpzc{Sh}_*(M)=M\otimes M.
\end{align*}
Moreover, the map $\shmap_*$ forms an algebra homomorphism.
\end{remark}

\begin{lem}[{\cite[(5.7)]{Sau}}]\label{lem:tensor is commu.}
For $M_1,M_2,N_1,N_2\in\overline{\BIMU}(\Gamma)$, we have
\begin{equation}\label{eqn:tensor is commu.}
	(M_1\otimes M_2)\times(N_1\otimes N_2)
	=(M_1\times N_1)\otimes(M_2\times N_2).
\end{equation}
\end{lem}
\begin{proof}
By the definition of the product $\times$, it is easy to prove this equation.
\end{proof}

\begin{prop}
The following two hold:
\footnote{In \cite[Proposition 5.1]{Sau}, it is proved that the pair $(\overline{\GARI}(\Gamma)_\as,\times)$ forms a subgroup of $(\overline{\GARI}(\Gamma),\times)$ and the pair $(\overline{\ARI}(\Gamma)_\al,[,])$ forms a Lie subalgebra of $(\overline{\ARI}(\Gamma),[,])$.}
\begin{enumerate}
\renewcommand{\labelenumi}{(\alph{enumi}).}
	\item The pairs $(\overline{\GARI}(\Gamma)_\as,\times)$ and $(\overline{\GARI}(\Gamma)_\is,\times)$ form subgroups of $(\overline{\GARI}(\Gamma),\times)$.
	\item The pairs $(\overline{\ARI}(\Gamma)_\al,[,])$ and $(\overline{\ARI}(\Gamma)_\il,[,])$ form Lie subalgebras of $(\overline{\ARI}(\Gamma),[,])$.
\end{enumerate}
\end{prop}
\begin{proof}
Let $M,N\in\overline{\GARI}(\Gamma)_\as$.
By using Lemma \ref{lem:shuffle map is alg. hom.} and Lemma \ref{lem:tensor is commu.}, we have
\begin{align*}
	\shmap(M\times N) = \shmap(M)\times \shmap(N)
	 = (M\times M)\otimes (N\times N) = (M\times N)\otimes (M\times N).
\end{align*}
By using Proposition \ref{prop:gp-like, Lie-like}.(ii), we get $M\times N\in\overline{\GARI}(\Gamma)_\as$.
Let $M,N\in \overline{\ARI}(\Gamma)_\al$.
Then we have
\begin{align*}
	&\shmap([M,N]) \\
	&= \shmap(M)\times \shmap(N)-\shmap(N)\times \shmap(M) \\
	&= (M\otimes I+ I\otimes M)\times (N\otimes I+ I\otimes N)
	- (N\otimes I+ I\otimes N)\times (M\otimes I+ I\otimes M) \\
	&= (M\times N)\otimes I + M\otimes N + N\otimes M + I\otimes (M\times N) \\
	&\quad- (N\times M)\otimes I - N\otimes M - M\otimes N - I\otimes (N\times M) \\
	&= [M,N]\otimes I + I\otimes [M,N].
\end{align*}
Hence, we get $[M,N]\in \overline{\ARI}(\Gamma)_\al$.
By using the map $\shmap_*$, we can prove that $(\overline{\GARI}(\Gamma)_\is,\times)$ forms a subgroup and $(\overline{\ARI}(\Gamma)_\il,[,])$ forms a Lie subalgebra similarly.
\end{proof}

\subsection{Exponential map $\exp_\times$}\label{subsec:Definition of exp}

\begin{definition}
We define the map $\exp_\times:\ARI(\Gamma) \rightarrow \GARI(\Gamma)$ by
\begin{equation}\label{eqn:exp}
	\exp_\times(A)
	:=\sum_{k\geq0} \frac{1}{k!}A^{\times k},
\end{equation}
for $A\in\ARI(\Gamma)$. Here, we use the notation $A^{\times k}:=\underbrace{A\times \cdots \times A}_{k}$.
\end{definition}
There is the inverse map $\log_\times:\GARI(\Gamma) \rightarrow \ARI(\Gamma)$ of map $\exp_\times$ defined by
\begin{equation}\label{eqn:log}
\log_\times(S):=\sum_{k\geq1}\frac{(-1)^{k-1}}{k}(S-I)^{\times k},
\end{equation}
for $S\in\GARI(\Gamma)$, and so we have $\exp_\times\circ\log_\times=\id_{\GARI(\Gamma)}$ and $\log_\times\circ\exp_\times=\id_{\ARI(\Gamma)}$.

We note that
\begin{equation}\label{eqn:exp on commutative case}
	\exp_\times(M+ N)=\exp_\times(M)\times \exp_\times(N)
\end{equation}
for $M,N\in\ARI(\Gamma)$ with $M\times N=N\times M$.
\begin{remark}
We also have the following exponential maps:
\renewcommand{\labelenumi}{(\alph{enumi}).}
\begin{enumerate}
	\item For $\overline{\ARI}(\Gamma)$ and $\overline{\GARI}(\Gamma)$, the similar exponential and the similar logarithm maps are defined.
	\item By replacing the product $\times$ of $\BIMU(\Gamma)$ in \eqref{eqn:exp} and \eqref{eqn:log} to the product of $\DIMU(\Gamma)$, the exponential and the logarithm maps of dimoulds are defined.
\end{enumerate}
For our simplicity, we denote these map by the same symbols $\exp_\times$ and $\log_\times$.
\end{remark}

\begin{thm}\label{thm:exp(ARIal)=GARIas etc}
The following three hold:
\footnote{The equation \eqref{eqn:exp:ARIal->GARIas} is proved in \cite[Proposition 5.1]{Sau}.}
\begin{align}\label{eqn:exp:ARIal->GARIas}
	&\exp_\times(\ARI(\Gamma)_\al)=\GARI(\Gamma)_\as, \\
	\label{eqn:exp:barARIal->barGARIas}
	&\exp_\times(\overline{\ARI}(\Gamma)_\al)=\overline{\GARI}(\Gamma)_\as, \\
	\label{eqn:exp:barARIil->barGARIis}
	&\exp_\times(\overline{\ARI}(\Gamma)_\il)=\overline{\GARI}(\Gamma)_\is.
\end{align}
\end{thm}
\begin{proof}
Because the above three are similarly shown, we only prove $\exp_\times(\ARI(\Gamma)_\al)=\GARI(\Gamma)_\as$.\\
($\subset$): Let $A\in\ARI(\Gamma)_\al$. Then we will show $\exp_\times(A)\in\GARI(\Gamma)_\as$.
By using Lemma \ref{lem:shuffle map is alg. hom.} and Proposition \ref{prop:gp-like, Lie-like}, we have
\begin{align*}
	\mathpzc{Sh}(\exp_\times(A))
	=& \exp_\times(\mathpzc{Sh}(A)) \\
	=& \exp_\times(A\otimes I+I\otimes A).
\intertext{Because we have $(A\otimes I)\times(I\otimes A)=(I\otimes A)\times(A\otimes I)$, by \eqref{eqn:exp on commutative case} we get}
	=& \exp_\times(A\otimes I)\times \exp(I\otimes A) \\
	=& (\exp_\times(A\otimes I)\times(I\otimes \exp_\times(A)) \\
	=& \exp_\times(A)\otimes \exp_\times(A).
\end{align*}
Hence, by Proposition \ref{prop:gp-like, Lie-like}, we obtain $\exp(A)\in\GARI(\Gamma)_\as$.\\
($\supset$): Let $S\in\GARI(\Gamma)_\as$. Similarly, we have
\begin{align*}
	\mathpzc{Sh}(\log_\times(S))
	=& \log_\times(\mathpzc{Sh}(S)) \\
	=& \log_\times(S\otimes S) \\
	=& \log_\times((I\otimes S)\times (S\otimes I)) \\
	=& \log_\times(I\otimes S)+\log_\times(S\otimes I) \\
	=& I\otimes \log_\times(S)+\log_\times(S)\otimes I.
\end{align*}
Hence, we get $\log_\times(S)\in\ARI(\Gamma)_\al$.
\end{proof}

\section{Automorphisms}\label{sec:Automorphism}

In this section, we recall the definition of the map $\ganit_v(B)$ (for $B\in\overline{\GARI}(\Gamma)$) introduced in \cite[(2.36)]{E-flex}, and we prove several properties of this map.
In \S \ref{subsec:Definition of ganit}, we review its definition (Definition \ref{def:ganit}) and we prove  that the map $\ganit_v(B)$ forms an algebra automorphism on $(\overline{\BIMU}(\Gamma),\times)$ (Proposition \ref{prop:alg aut of ganit B}), and that this map induces a group automorphism $(\overline{\GARI}(\Gamma),\times)$ and induces a Lie algebra automorphism on $(\overline{\ARI}(\Gamma),[,])$ (Theorem \ref{thm: ganit(B) is automorphism}).
As a corollary, we get a commutativity of $\ganit_v(\pic)$ and $\exp_\times$ (Corollary \ref{cor:commutative diagram}).
In \S \ref{subsec:Algebraic preparation}, we introduce elements $g_B(\vecv_r) \in \mathcal K\langle V_\Z \rangle$ (Definition \ref{def:definition of map g}) as an analogue of $\ganit_v(B)$.
In \S \ref{subsec:map between alternal and alternil}, we prove the recurrence formula (Proposition \ref{prop:recurrence formula of map g}) of $g_B(\vecv_r)$ for $B=\pic$.
By using this recurrence formula, we prove Theorem \ref{thm: ganit(B) is automorphism between al and il or as and is}, that is, the map $\ganit_v(B)$ induces a group isomorphism from $\overline{\GARI}(\Gamma)_\as$ to $\overline{\GARI}(\Gamma)_\is$ and induces a Lie algebra isomorphism from $\overline{\ARI}(\Gamma)_\al$ to $\overline{\ARI}(\Gamma)_\il$.
As a corollary, we obtain its commutativity (Corollary \ref{cor:commutative diagram of al,il,as,is}) with the exponential map $\exp_\times$ and we obtain a bijection between $\ARI(\Gamma)_{\underline{\al}*\underline{\al}}$ and $\ARI(\Gamma)_{\underline{\al}*\underline{\il}}$ (Theorem \ref{thm:adari(pal) induce bijection}).

\subsection{Definition of $\ganit_v(B)$}\label{subsec:Definition of ganit}

Hereafter, for our simplicity, we denote $\omega_i:=\binom{\sigma_i}{v_i}$ for $i\geq1$ and we denote $\vecv_r=(\omega_1,\dots,\omega_r)$ for $r\geq1$.
\begin{definition}[{\cite[(2.36)]{E-flex}, \cite[(2.7.2)]{S-ARIGARI}}]\label{def:ganit}
Let $B\in\overline{\GARI}(\Gamma)$. We define the $\Q$-linear map $\ganit_v(B):\overline{\BIMU}(\Gamma) \rightarrow \overline{\BIMU}(\Gamma)$ by $(\ganit_v(B)(A))(\vecv_0):=A(\emptyset)$ and 
\begin{align}\label{eqn:def of ganit}
	(\ganit_v(B)(A))(\vecv_r):=\sum_{s\geq1}\sum_{\substack{
		\vecv_r=a_1b_1\cdots a_sb_s \\
		a_1,\dots,a_s\neq\emptyset \\
		b_1,\dots,b_{s-1}\neq\emptyset}}
	B(\lrflex{a_1}{b_1})\cdots B(\lrflex{a_s}{b_s})
	\ A(\ulflex{a_1}{b_1}\cdots \ulflex{a_s}{b_s}),
\end{align}
for $r\geq1$ and for $A\in\overline{\BIMU}(\Gamma)$.
\end{definition}

\begin{example}
Because the definition of $\ganit_v(B)(A)$ is complicated, we explicitly give some examples of $\ganit_v(B)(A)$.
\begin{align*}
	&(\ganit_v(B)(A))(\vecv_1)=A(\omega_1), \\
	&(\ganit_v(B)(A))(\vecv_2)
	=A(\omega_1,\omega_2)+B(\lrflex{\omega_1}{\omega_2})A(\ulflex{\omega_1}{\omega_2}), \\
	&(\ganit_v(B)(A))(\vecv_3) 
	=A(\omega_1,\omega_2,\omega_3)
		+B(\lrflex{\omega_2}{\omega_3})A(\ulflex{\omega_1,\omega_2}{\omega_3}) \\
	&\hspace{3.5cm}+B(\lrflex{\omega_1}{\omega_2})A(\ulflex{\omega_1}{\omega_2},\omega_3)
		+B(\lrflex{\omega_1}{(\omega_2,\omega_3)})A(\ulflex{\omega_1}{(\omega_2,\omega_3)}), \\
	&(\ganit_v(B)(A))(\vecv_4) \\
	&=A(\omega_1,\omega_2,\omega_3,\omega_4)
		+B(\lrflex{\omega_3}{\omega_4})A(\ulflex{\omega_1,\omega_2,\omega_3}{\omega_4}) \\
	&\quad+B(\lrflex{\omega_2}{\omega_3})A(\ulflex{\omega_1,\omega_2}{\omega_3},\omega_4)
		+B(\lrflex{\omega_2}{(\omega_3,\omega_4)})A(\ulflex{\omega_1,\omega_2}{(\omega_3,\omega_4)}) \\
	&\quad\quad\quad+B(\lrflex{\omega_1}{\omega_2})A(\ulflex{\omega_1}{\omega_2},\omega_3,\omega_4)
		+B(\lrflex{\omega_1}{\omega_2})B(\lrflex{\omega_3}{\omega_4})A(\ulflex{\ulflex{\omega_1}{\omega_2},\omega_3}{\omega_4}) \\
	&\quad\quad\quad\quad+B(\lrflex{\omega_1}{(\omega_2,\omega_3)})A(\ulflex{\omega_1}{(\omega_2,\omega_3)},\omega_4)
		+B(\lrflex{\omega_1}{(\omega_2,\omega_3,\omega_4)})A(\ulflex{\omega_1}{(\omega_2,\omega_3,\omega_4)}).
\end{align*}
\end{example}

\begin{remark}
Let $B\in\overline{\GARI}(\Gamma)$. By Definition \ref{def:ganit}, we have
\begin{align*}
	&\ganit_v(B)(A)(\emptyset)=A(\emptyset)=0, \\
	&\ganit_v(B)(S)(\emptyset)=S(\emptyset)=1,
\end{align*}
for $A\in\overline{\ARI}(\Gamma)$ and $S\in\overline{\GARI}(\Gamma)$.
Therefore, $\ganit_v(B)$ induces the $\Q$-linear maps on $\overline{\ARI}(\Gamma)$ and on $\overline{\GARI}(\Gamma)$.
\end{remark}

\begin{remark}\label{rem:ganit induce map on ARI and GARI}
Let $B,C\in\overline{\GARI}(\Gamma)$. Assume that we have $\ganit_v(B)(\ganit_v(C)(A))=A$ for any $A\in\overline{\BIMU}(\Gamma)$.
Then by the above examples, we get
\begin{align*}
	A(\omega_1,\omega_2)+\left\{B(\lrflex{\omega_1}{\omega_2})+C(\lrflex{\omega_1}{\omega_2})\right\}A(\ulflex{\omega_1}{\omega_2})
	=A(\omega_1,\omega_2).
\end{align*}
Therefore, we obtain $C(\omega_1)=-B(\omega_1)$. In general, by computing 
$$
\ganit_v(B)(\ganit_v(C)(A))=A,
$$
we see that the component $C(\vecv_r)$ can be algebraically expressed by using $B(\vecv_m)$ ($1\leq  m\leq r$).
Hence, for any $B\in\overline{\GARI(\Gamma)}$, there is the inverse map of $\ganit_v(B)$.

In \cite[\S 4.7]{E-flex}, it is mentioned that the inverse map of $\ganit_v(\pic)$ is given by
$$
\ganit_v(\pari\circ\anti(\pij)),
$$
where the map $\pari$ and $\anti$ are defined by
\begin{align*}
&\pari(M)^r\varia{\sigma_1,\ \dots,\ \sigma_r}{v_1,\ \dots,\ v_r}
:=(-1)^rM^r\varia{\sigma_1,\ \dots,\ \sigma_r}{v_1,\ \dots,\ v_r}, \\
&\anti(M)^r\varia{\sigma_1,\ \dots,\ \sigma_r}{v_1,\ \dots,\ v_r}
:=M^r\varia{\sigma_r,\ \dots,\ \sigma_1}{v_r,\ \dots,\ v_1}.
\end{align*}
In fact, by direct calculation, we can check that $\ganit_v(\pari\circ\anti(\pij))$ is the inverse map of $\ganit_v(\pic)$.
In \cite{B} and \cite{S-ARIGARI}, the mould $\poc\in\overline{\GARI}(\Gamma)$ is introduced by
\begin{equation}\label{eqn:def of poc}
\poc:=\pari\circ\anti(\pij),
\end{equation}
that is, $\poc(\vecv_0):=1$ and $\poc(\vecv_1):=-\frac{1}{v_1}$ and
$$
\poc(\vecv_r):=-\frac{1}{v_1(v_1-v_2)(v_2-v_3)\cdots(v_{r-1}-v_r)}.
$$
By using this mould $\poc$, the inverse map of $\ganit_v(\pic)$ is denoted by $\ganit_v(\poc)$.
\end{remark}

\begin{prop}\label{prop:alg aut of ganit B}
For any $B\in\overline{\GARI}(\Gamma)$, the map $\ganit_v(B)$ forms an algebra automorphism on $(\overline{\BIMU}(\Gamma),\times)$, that is, for any $A_1,A_2\in\overline{\BIMU}(\Gamma)$, we have
$$
\ganit_v(B)(A_1\times A_2)
=\ganit_v(B)(A_1)\times\ganit_v(B)(A_2).
$$
\end{prop}

We use the following lemma to prove this proposition.
\begin{lem}\label{lem: lemma for alg aut of ganit B}
Let $A_1,A_2\in\overline{\BIMU}(\Gamma)$ and $s\geq1$. For $a_1,\dots,a_s\in V_\Z^\bullet\setminus\{\emptyset\}$ and for $b_1,\dots,b_{s-1}\in V_\Z^\bullet\setminus\{\emptyset\}$ and for $b_s\in V_\Z^\bullet$, we put $\omega=\ulflex{a_1}{b_1} \cdots \ulflex{a_s}{b_s}$. Then we have
\begin{align*}
	&(A_1\times A_2)(\omega) \\
	&=A_1(\omega)A_2(\emptyset) + A_1(\emptyset)A_2(\omega)
	+\sum_{p=1}^{s-1}
	A_1(\ulflex{a_1}{b_1}\cdots \ulflex{a_p}{b_p})
	A_2(\ulflex{a_{p+1}}{b_{p+1}}\cdots \ulflex{a_s}{b_s}) \\
	&\quad+ \sum_{p=1}^s\sum_{\substack{
		\ulflex{a_p}{b_p}=\alpha\beta \\
		\alpha,\beta\neq\emptyset}}
	A_1(\ulflex{a_1}{b_1}\cdots \ulflex{a_{p-1}}{b_{p-1}}\alpha)
	A_2(\beta\ulflex{a_{p+1}}{b_{p+1}}\cdots \ulflex{a_s}{b_s}).
\end{align*}
\end{lem}
\begin{proof}
By definition of the product $\times$, we have
\begin{align*}
	&(A_1\times A_2)(\omega)
	=\sum_{\omega=\alpha\beta}A_1(\alpha)A_2(\beta) \\
	&=A_1(\omega)A_2(\emptyset) 
	+\sum_{p=1}^s\sum_{\substack{
		\ulflex{a_p}{b_p}=\alpha\beta \\
		\beta\neq\emptyset}}
	A_1(\ulflex{a_1}{b_1}\cdots \ulflex{a_{p-1}}{b_{p-1}}\alpha)
	A_2(\beta\ulflex{a_{p+1}}{b_{p+1}}\cdots \ulflex{a_s}{b_s}) \\
	&=A_1(\omega)A_2(\emptyset) 
	+\sum_{p=1}^s
	A_1(\ulflex{a_1}{b_1}\cdots \ulflex{a_{p-1}}{b_{p-1}})
	A_2(\ulflex{a_p}{b_p}\ulflex{a_{p+1}}{b_{p+1}}\cdots \ulflex{a_s}{b_s}) \\
	&\quad +\sum_{p=1}^s\sum_{\substack{
		\ulflex{a_p}{b_p}=\alpha\beta \\
		\alpha,\beta\neq\emptyset}}
	A_1(\ulflex{a_1}{b_1}\cdots \ulflex{a_{p-1}}{b_{p-1}}\alpha)
	A_2(\beta\ulflex{a_{p+1}}{b_{p+1}}\cdots \ulflex{a_s}{b_s}).
\end{align*}
The case of $p=1$ in the second term coincides with $A_1(\emptyset)A_2(\omega)$.
Hence, we obtain the claim.
\end{proof}
By using Lemma \ref{lem: lemma for alg aut of ganit B}, we prove Proposition \ref{prop:alg aut of ganit B}.\\

\noindent
{\it Proof of Proposition \ref{prop:alg aut of ganit B}}.
Let $B\in\overline{\GARI}(\Gamma)$ and $A_1,A_2\in\overline{\BIMU}(\Gamma)$.
When $r=0,1$, it is easy to get
$$
(\ganit_v(B)(A_1\times A_2))(\vecv_r)=(A_1\times A_2)(\vecv_r)
=(\ganit_v(B)(A_1)\times\ganit_v(B)(A_2))(\vecv_r).
$$
For $r\geq2$, by definition we have
\begin{align*}
	&(\ganit_v(B)(A_1\times A_2))(\vecv_r) \\
	&=\sum_{s\geq1}\sum_{\substack{
		\vecv_r=a_1b_1\cdots a_sb_s \\
		a_1,\dots,a_s\neq\emptyset \\
		b_1,\dots,b_{s-1}\neq\emptyset}}
	B(\lrflex{a_1}{b_1})\cdots B(\lrflex{a_s}{b_s})
	\ (A_1\times A_2)(\ulflex{a_1}{b_1}\cdots \ulflex{a_s}{b_s}).
\end{align*}
By using Lemma \ref{lem: lemma for alg aut of ganit B}, we calculate
\begin{align}\label{eqn: ganit(B)(A1 times A2)=}
	&(\ganit_v(B)(A_1\times A_2))(\vecv_r) \\
	&=(\ganit(B)(A_1))(\vecv_r)A_2(\emptyset) + A_1(\emptyset)(\ganit(B)(A_2))(\vecv_r) \nonumber\\
	&\quad +\sum_{s\geq1}\sum_{\substack{
		\vecv_r=a_1b_1\cdots a_sb_s \\
		a_1,\dots,a_s\neq\emptyset \\
		b_1,\dots,b_{s-1}\neq\emptyset}}
	B(\lrflex{a_1}{b_1})\cdots B(\lrflex{a_s}{b_s})
	\ \left\{\sum_{p=1}^{s-1}
	A_1(\ulflex{a_1}{b_1}\cdots \ulflex{a_p}{b_p})
	A_2(\ulflex{a_{p+1}}{b_{p+1}}\cdots \ulflex{a_s}{b_s}) \right. \nonumber\\
	&\quad\quad \left.+ \sum_{p=1}^s\sum_{\substack{
		\ulflex{a_p}{b_p}=\alpha\beta \\
		\alpha,\beta\neq\emptyset}}
	A_1(\ulflex{a_1}{b_1}\cdots \ulflex{a_{p-1}}{b_{p-1}}\alpha)
	A_2(\beta\ulflex{a_{p+1}}{b_{p+1}}\cdots \ulflex{a_s}{b_s})\right\}. \nonumber
\end{align}
Here, by putting $q=s-p$, we calculate the third term of the right hand side of \eqref{eqn: ganit(B)(A1 times A2)=} as
\begin{align*}
	&\sum_{s\geq1}\sum_{\substack{
		\vecv_r=a_1b_1\cdots a_sb_s \\
		a_1,\dots,a_s\neq\emptyset \\
		b_1,\dots,b_{s-1}\neq\emptyset}}
	B(\lrflex{a_1}{b_1})\cdots B(\lrflex{a_s}{b_s})
	\ \sum_{p=1}^{s-1}
	A_1(\ulflex{a_1}{b_1}\cdots \ulflex{a_p}{b_p})
	A_2(\ulflex{a_{p+1}}{b_{p+1}}\cdots \ulflex{a_s}{b_s}) \\
	&=\sum_{p,q\geq1}\sum_{\substack{
		\vecv_r=a_1b_1\cdots a_{p+q}b_{p+q} \\
		a_1,\dots,a_{p+q}\neq\emptyset \\
		b_1,\dots,b_{p-1}\neq\emptyset \\
		b_{p+1},\dots,b_{p+q-1}\neq\emptyset \\
		b_p\neq\emptyset}}
	B(\lrflex{a_1}{b_1})\cdots B(\lrflex{a_{p+q}}{b_{p+q}})
	\ A_1(\ulflex{a_1}{b_1}\cdots \ulflex{a_p}{b_p})
	A_2(\ulflex{a_{p+1}}{b_{p+1}}\cdots \ulflex{a_{p+q}}{b_{p+q}}).
\end{align*}
While, we put $q=s-p+1$ and we replace
\begin{align*}
	& a_i \longrightarrow a_{i+1}, \qquad (p+1 \leqslant i \leqslant s) \\
	& b_j \longrightarrow b_{j+1}, \qquad (p \leqslant j \leqslant s)
\end{align*}
and replace $\alpha=a_p$, $\beta=\ulflex{a_{p+1}}{b_p+1}$.
Then we calculate the fourth term of the right hand side of \eqref{eqn: ganit(B)(A1 times A2)=} as
\begin{align*}
	&\sum_{s\geq1}\sum_{\substack{
		\vecv_r=a_1b_1\cdots a_sb_s \\
		a_1,\dots,a_s\neq\emptyset \\
		b_1,\dots,b_{s-1}\neq\emptyset}}
	B(\lrflex{a_1}{b_1})\cdots B(\lrflex{a_s}{b_s})
	\ \sum_{p=1}^s\sum_{\substack{
		\ulflex{a_p}{b_p}=\alpha\beta \\
		\alpha,\beta\neq\emptyset}}
	A_1(\ulflex{a_1}{b_1}\cdots \ulflex{a_{p-1}}{b_{p-1}}\alpha)
	A_2(\beta\ulflex{a_{p+1}}{b_{p+1}}\cdots \ulflex{a_s}{b_s}) \\
	&=\sum_{p,q\geq1}\sum_{\substack{
		\vecv_r=a_1b_1\cdots a_{p+q}b_{p+q} \\
		a_1,\dots,a_{p+q}\neq\emptyset \\
		b_1,\dots,b_{p-1}\neq\emptyset \\
		b_{p+1},\dots,b_{p+q-1}\neq\emptyset \\
		b_p=\emptyset}}
	B(\lrflex{a_1}{b_1})\cdots B(\lrflex{a_{p+q}}{b_{p+q}})
	\ A_1(\ulflex{a_1}{b_1}\cdots \ulflex{a_p}{b_p})
	A_2(\ulflex{a_{p+1}}{b_{p+1}}\cdots \ulflex{a_{p+q}}{b_{p+q}}).
\end{align*}
Hence, we get
\begin{align*}
	&(\ganit_v(B)(A_1\times A_2))(\vecv_r) \\
	&=(\ganit(B)(A_1))(\vecv_r)A_2(\emptyset) + A_1(\emptyset)(\ganit(B)(A_2))(\vecv_r) \\
	&\quad +\sum_{p,q\geq1}
	\sum_{\substack{
		\vecv_r=a_1b_1\cdots a_{p+q}b_{p+q} \\
		a_1,\dots,a_{p+q}\neq\emptyset \\
		b_1,\dots,b_{p-1}\neq\emptyset \\
		b_{p+1},\dots,b_{p+q-1}\neq\emptyset}}
	B(\lrflex{a_1}{b_1})\cdots B(\lrflex{a_{p+q}}{b_{p+q}}) \\
	&\hspace{5.3cm} \cdot A_1(\ulflex{a_1}{b_1}\cdots \ulflex{a_p}{b_p})
	\ A_2(\ulflex{a_{p+1}}{b_{p+1}}\cdots \ulflex{a_{p+q}}{b_{p+q}}).
\intertext{By factoring the summation of the third term, we calculate}
	&=(\ganit(B)(A_1))(\vecv_r)A_2(\emptyset) + A_1(\emptyset)(\ganit(B)(A_2))(\vecv_r) \\
	&\quad +\sum_{\substack{
		\vecv_r=\alpha\beta \\
		\alpha,\beta\neq\emptyset}}
	\left\{
	\sum_{p\geq1}\sum_{\substack{
		\alpha=a_1b_1\cdots a_pb_p \\
		a_1,\dots,a_p\neq\emptyset \\
		b_1,\dots,b_{p-1}\neq\emptyset}}
	B(\lrflex{a_1}{b_1})\cdots B(\lrflex{a_p}{b_p})
	\ A_1(\ulflex{a_1}{b_1}\cdots \ulflex{a_p}{b_p})
	\right\} \\
	&\quad \cdot \left\{
	\sum_{q\geq1}\sum_{\substack{
		\beta=a_{p+1}b_{p+1}\cdots a_{p+q}b_{p+q} \\
		a_{p+1},\dots,a_{p+q}\neq\emptyset \\
		b_{p+1},\dots,b_{p+q-1}\neq\emptyset}}
	B(\lrflex{a_{p+1}}{b_{p+1}})\cdots B(\lrflex{a_{p+q}}{b_{p+q}})
	\ A_2(\ulflex{a_{p+1}}{b_{p+1}}\cdots \ulflex{a_{p+q}}{b_{p+q}})
	\right\} \\
	&=\sum_{\vecv_r=\alpha\beta}(\ganit_v(B)(A_1))(\alpha)(\ganit_v(B)(A_2))(\beta) \\
	&=(\ganit_v(B)(A_1)\times\ganit_v(B)(A_2))(\vecv_r),
\end{align*}
Hence, we finish the proof.
\hfill $\Box$

\begin{thm}\label{thm: ganit(B) is automorphism}
For $B\in\overline{\GARI}(\Gamma)$, the $\Q$-linear map $\ganit_v(B)$ induces a group automorphism on $(\overline{\GARI}(\Gamma),\times)$ and induces a Lie algebra automorphism on $(\overline{\ARI}(\Gamma),[,])$.
\end{thm}
\begin{proof}
By combining Remark \ref{rem:ganit induce map on ARI and GARI} and Proposition \ref{prop:alg aut of ganit B}, we see that the map $\ganit_v(B)$ induces a group homomorphism on $(\overline{\GARI}(\Gamma),\times)$ and induces a Lie algebra homomorphism on $(\overline{\ARI}(\Gamma),[,])$.
Because the map $\ganit_v(B)$ has the inverse map by Remark \ref{rem:ganit induce map on ARI and GARI}, we get the claim.
\end{proof}

\begin{cor}\label{cor:commutative diagram}
For $B\in\overline{\GARI}(\Gamma)$, the following diagram commutes:
$$
\xymatrix{
\overline{\GARI}(\Gamma) \ar[rr]^{\ganit_v(B)} & & \overline{\GARI}(\Gamma)  \\
\overline{\ARI}(\Gamma) \ar[rr]_{\ganit_v(B)} \ar@{->}[u]^{\exp_\times}& & \overline{\ARI}(\Gamma) \ar@{->}[u]_{\exp_\times}
}
$$
\end{cor}
\begin{proof}
Let $B\in\overline{\GARI}(\Gamma)$ and $A\in \overline{\ARI}(\Gamma)$.
By using the definition \eqref{eqn:exp} of the map $\exp_\times$ and Theorem \ref{thm: ganit(B) is automorphism}, we have
\begin{align*}
\ganit_v(B)\left( \exp_\times(A) \right)
&=\ganit_v(B)\left( \sum_{k\geq0} \frac{1}{k!}A^{\times k} \right)
	= \sum_{k\geq0} \frac{1}{k!}\ganit_v(B)\left(A^{\times k} \right) \\
&= \sum_{k\geq0} \frac{1}{k!}\ganit_v(B)\left(A \right)^{\times k}
	=\exp_\times\left( \ganit_v(B)(A) \right).
\end{align*}
Hence, we obtain the above diagram.
\end{proof}

\subsection{Reformulation of $\ganit_v(B)$}\label{subsec:Algebraic preparation}
In this subsection, we introduce elements $g_B(\vecv_r) \in \mathcal K\langle V_\Z \rangle$ (Definition \ref{def:definition of map g}) as an analogue of $\ganit_v(B)$.
We prove Lemma \ref{lem:equivalent equation of g}, which is essential to prove several theorems in \S \ref{subsec:map between alternal and alternil}.

\begin{definition}\label{def:definition of map g}
For $B\in\overline{\GARI}(\Gamma)$ and $r\geq0$, we define $g_B(\vecv_r)\in\mathcal K \langle V_\Z\rangle$ by $g_B(\vecv_0):=\emptyset$ and for $r\geq1$
\begin{align}\label{eqn:definition of map g}
	g_B(\vecv_r):=\sum_{s\geq1}\sum_{\substack{
		\vecv_r=a_1b_1\cdots a_sb_s \\
		a_1,\dots,a_s\neq\emptyset \\
		b_1,\dots,b_{s-1}\neq\emptyset}}
	B(\lrflex{a_1}{b_1})\cdots B(\lrflex{a_s}{b_s})
	\ \bigl(\ulflex{a_1}{b_1}\cdots \ulflex{a_s}{b_s}\bigr).
\end{align}
\end{definition}
\begin{remark}
The above definition \eqref{eqn:definition of map g} is an analogue of \eqref{eqn:def of ganit}.
In fact, for $r=1,2,3$, we have
\begin{align*}
	g_B(\vecv_1)&=(\omega_1), \\
	g_B(\vecv_2)
	&=(\omega_1,\omega_2)+B(\lrflex{\omega_1}{\omega_2})(\ulflex{\omega_1}{\omega_2}), \\
	g_B(\vecv_3) 
	&=(\omega_1,\omega_2,\omega_3)
		+B(\lrflex{\omega_2}{\omega_3})(\ulflex{\omega_1,\omega_2}{\omega_3}) \\
	&\quad +B(\lrflex{\omega_1}{\omega_2})(\ulflex{\omega_1}{\omega_2},\omega_3)
		+B(\lrflex{\omega_1}{(\omega_2,\omega_3)})(\ulflex{\omega_1}{(\omega_2,\omega_3)}).
\end{align*}
This $g_B$ is useful to prove Theorem \ref{thm: ganit(B)(ARIal) subset ARIil}.
\end{remark}
\begin{remark}
Assume that $u_1,\dots,u_r\in \mathcal F$ are algebraically independent over $\Q$, and put $\alpha_i:=\varia{\sigma_i}{u_i}$ for $1\leq i\leq r$.
We denote $g_B(\alpha_1,\dots,\alpha_r)$ to be the image of $g_B(\vecv_r)$ under the field embedding $\Q(v_1,\dots,v_r) \hookrightarrow \mathcal F$ sending $v_i\mapsto u_i$.
For $c_1,\dots,c_m\in\mathcal K $ and for $\beta_1,\dots,\beta_m\in V_\Z^\bullet$, we also denote
$$
g_B( c_1\beta_1 + \cdots + c_m\beta_m ):=c_1g_B(\beta_1) + \cdots + c_mg_B(\beta_m).
$$
\end{remark}

We introduce several notations which are used to prove claims of this subsection.
\begin{notation}\label{not:index symbols}
In Notation \ref{not:on expression of elements in V_Z}, we denote the element $u_1\cdots u_m\in V_\Z^\bullet$ by $(u_1,\dots,u_m)$ for $u_1,\dots,u_m\in V_\Z$.
To avoid confusion, we sometimes denote the element $(c_1,\dots,c_t)\in(V_{\Z}^\bullet)^t$ by
$(c_1;\cdots;c_t)$ for $c_1,\dots,c_t\in V_\Z^\bullet$.
By using notation, we denote
$$
W_B(\vecu):=W_B(a_1;b_1;\cdots;a_s;b_s)
:=B(\lrflex{a_1}{b_1})\cdots B(\lrflex{a_s}{b_s})
	\ \bigl(\ulflex{a_1}{b_1}\cdots \ulflex{a_s}{b_s}\bigr),
$$
for $s\geq1$ and $\vecu=(a_1;b_1;\cdots;a_s;b_s)\in (V_\Z^\bullet)^{2s}$ and $B\in\overline{\GARI}(\Gamma)$.
Because we have $B(\lrflex{u}{\emptyset})=1$ and $\ulflex{(u\omega)}{\eta}=(u)(\ulflex{\omega}{\eta})$ for $u\in V_\Z$ and $\omega\in V_\Z^\bullet\setminus\{\emptyset\}$, it is easy to show
\begin{equation}\label{eqn:W(uomega;vecu)=W(u;empty;omega;vecu)}
W_B(u\omega;\vecu)=W_B(u;\emptyset;\omega;\vecu).
\end{equation}
\end{notation}

\begin{definition}
For any word $\omega\in V_{\Z}^\bullet$ and for $t\in\N$, we define two sets $D_t(\omega)$ and $E_t(\omega)$ consisting of decompositions of $\omega$ by
\begin{align*}
&D_t(\omega)
:=\left\{ (c_1;\cdots;c_t)\in (V_{\Z}^\bullet)^t\ |\ \omega=c_1\cdots c_t,\ c_1,\dots,c_{t-1}\neq\emptyset \right\}, \\
&E_t(\omega)
:=\left\{ (c_1;\cdots;c_t)\in (V_{\Z}^\bullet)^t\ |\ \omega=c_1\cdots c_t,\ c_2,\dots,c_{t-1}\neq\emptyset \right\}.
\end{align*}
For $t\geq2$, we define two subsets $D_t^{\geq2}(\omega)$ and $D_t^1(\omega)$ of $D_t(\omega)$ by
\begin{align*}
D_t^{\geq2}(\omega)
&:=\left\{ (c_1;\cdots;c_t)\in D_t(\omega)\ |\ l(c_1)\geq2 \right\}, \\
D_t^1(\omega)
&:=\left\{ (c_1;\cdots;c_t)\in D_t(\omega)\ |\ l(c_1)=1 \right\}.
\end{align*}
Note that $E_1(\omega)=D_1(\omega)$.
\end{definition}

\begin{remark}
By the above notation and definition, we have
\begin{align}\label{eqn:definition of map g, by using D}
g_B(\vecv_r)
=\sum_{s\geq1}\sum_{\vecu \in D_{2s}(\vecv_r)}
	W_B(\vecu).
\end{align}
\end{remark}

It is clear to show the following lemma.
\begin{lem}
For $t\geq2$, we have 
\begin{equation}\label{eqn:partition of D}
D_t(\omega)=D_t^{\geq2}(\omega)\,\sqcup\, D_t^1(\omega),
\end{equation}
that is, a family of sets $\{D_t^{\geq2}(\omega), D_t^1(\omega)\}$ is a partition of $D_t(\omega)$, and we have 
\begin{equation}\label{eqn:partition of E}
E_t(\omega)=(\{\emptyset\}\times D_{t-1}(\omega))\sqcup D_t(\omega),
\end{equation}
that is, a family of sets $\{\{\emptyset\}\times D_{t-1}(\omega), D_t(\omega)\}$ is a partition of $E_t(\omega)$.
\end{lem}
\begin{proof}
These partitions follow from the definitions.
\end{proof}

For $r\geq1$ and $\omega=(\alpha_1,\dots,\alpha_r)$ with $\alpha_1,\dots,\alpha_r\in V_{\Z}$, we sometimes denote
$$
\omega':=\left\{\begin{array}{ll}
	(\alpha_2,\dots,\alpha_r) & (r\geq2), \\
	\emptyset & (r=1).
\end{array}\right.
$$
By using these symbols, we have
\begin{align*}
D_t^{\geq2}(\omega)
&=\left\{ (\alpha_1c_1;c_2;\cdots;c_t)\ |\ \omega'=c_1\cdots c_t,\ c_1,\dots,c_{t-1}\neq\emptyset \right\}, \\
D_t^1(\omega)
&=\left\{ (\alpha_1;c_2;\cdots;c_t)\ |\ \omega'=c_2\cdots c_t,\ c_2,\dots,c_{t-1}\neq\emptyset \right\},
\end{align*}
\begin{lem}
Let $\omega=(\alpha_1,\dots,\alpha_r)$ with $\alpha_1,\dots,\alpha_r\in V_{\Z}$.
Then there exist two bijections
\begin{align}
\label{eqn:bijection from Dgeq2 to D}
D_t^{\geq2}(\omega) \longrightarrow D_t(\omega')
&\ ;\ (\alpha_1c_1;c_2;\cdots;c_t) \longmapsto (c_1;c_2;\cdots;c_t), \\
\label{eqn:bijection from D1 to D}
D_t^1(\omega) \longrightarrow D_{t-1}(\omega')
&\ ;\ (\alpha_1;c_2;\cdots;c_t) \longmapsto (c_2;\cdots;c_t),
\end{align}
for $t\geq2$, and there exists a bijection
\begin{align}\label{eqn:bijection from E to D}
E_t(\omega') &\longrightarrow D_t(\omega_1,\omega')=D_t(\omega) \\
(b_0;\vecu) &\longmapsto \ (\alpha_1b_0;\vecu) \nonumber
\end{align}
for $t\geq2$.
\end{lem}
\begin{proof}
These three bijections follow from the definitions.
\end{proof}

\begin{lem}\label{lem:equivalent equation of g}
For $r\geq 1$, we have
\begin{align*}
g_B(\vecv_r)
=\sum_{s\geq0}\sum_{(b_0;\vecu)\in E_{2s+1}(\vecv_r')}
	B(\lrflex{\omega_1}{b_0})\ (\ulflex{\omega_1}{b_0})
	W_B(\vecu).
\end{align*}
\end{lem}
\begin{proof}
By using the bijection \eqref{eqn:bijection from E to D} for $t=2s$ and $\omega=\vecv_r$, we have
\begin{align*}
g_B(\vecv_r)
&=\sum_{s\geq1}\sum_{(a_1;b_1;\vecu) \in D_{2s}(\vecv_r)}
	W_B(a_1;b_1;\vecu)
=\sum_{s\geq1}\sum_{(a_1;b_1;\vecu) \in E_{2s}(\vecv_r')}
	W_B(\omega_1a_1;b_1;\vecu).
\intertext{By using the partition \eqref{eqn:partition of E} for $t=2s$ and $\omega=\vecv_r'$, we get}
&=\sum_{s\geq1}\sum_{(b_1;\vecu) \in D_{2s-1}(\vecv_r')}
	W_B(\omega_1;b_1;\vecu)
+\sum_{s\geq1}\sum_{(a_1;b_1;\vecu) \in D_{2s}(\vecv_r')}
	W_B(\omega_1a_1;b_1;\vecu).
\intertext{By changing variables of the first summation and by applying \eqref{eqn:W(uomega;vecu)=W(u;empty;omega;vecu)} to the second summation, we calculate}
&=\sum_{s\geq0}\sum_{(b_0;\vecu) \in D_{2s+1}(\vecv_r')}
	W_B(\omega_1;b_0;\vecu)
+\sum_{s\geq1}\sum_{(a_1;b_1;\vecu) \in D_{2s}(\vecv_r')}
	W_B(\omega_1;\emptyset;a_1;b_1;\vecu) \\
&=W_B(\omega_1;\vecv_r')
+\sum_{s\geq1}\sum_{(b_0;\vecu) \in D_{2s+1}(\vecv_r')}
	W_B(\omega_1;b_0;\vecu)
+\sum_{s\geq1}\sum_{(b_0;\vecu) \in \{\emptyset\}\times D_{2s}(\vecv_r')}
	W_B(\omega_1;b_0;\vecu).
\intertext{By using the partition \eqref{eqn:partition of E} again, we have}
&=W_B(\omega_1;\vecv_r')
+\sum_{s\geq1}\sum_{(b_0;\vecu) \in E_{2s+1}(\vecv_r')}
	W_B(\omega_1;b_0;\vecu) \\
&=\sum_{s\geq0}\sum_{(b_0;\vecu) \in E_{2s+1}(\vecv_r')}
	\pic(\lrflex{\omega_1}{b_0})\ (\ulflex{\omega_1}{b_0})W_B(\vecu).
\end{align*}
Hence, we obtain the claim.
\end{proof}

\subsection{Automorphisms between $\overline{\ARI}(\Gamma)_\al$ and $\overline{\ARI}(\Gamma)_\il$ and between $\overline{\GARI}(\Gamma)_\as$ and $\overline{\GARI}(\Gamma)_\is$}\label{subsec:map between alternal and alternil}

In this subsection, we prove the recurrence formula (Proposition \ref{prop:recurrence formula of map g}) of $g_B(\vecv_r)$ for $B=\pic$.
By using this recurrence formula, we prove Theorem \ref{thm: ganit(B) is automorphism between al and il or as and is}, that is, the map $\ganit_v(B)$ induces a group isomorphism from $\overline{\ARI}(\Gamma)_\al$ to $\overline{\ARI}(\Gamma)_\il$ and induces a Lie algebra isomorphism from $\overline{\GARI}(\Gamma)_\as$ to $\overline{\GARI}(\Gamma)_\is$.
As a corollary, we obtain a commutative diagram (Corollary \ref{cor:commutative diagram of al,il,as,is}).

We consider the map $\ganit_v(B)$ in the case of $B=\pic$ defined in Example \ref{ex:al,il,as,is mould}.(d).
For our simplicity, we often denote $g_{\pic}$ by $g$ and denote $W_{\pic}(\vecu)$ by $W(\vecu)$.

We prove the following key recurrence formulas of $g$, which is important to prove Theorem \ref{thm: ganit(B) is automorphism between al and il or as and is}.
\begin{prop}\label{prop:recurrence formula of map g}
For $r\geq2$, we have
\begin{align}\label{eqn:recurrence formula of map g}
g(\vecv_r)
&=(\omega_1)g(\vecv_r')
+\pic(\lrflex{\omega_1}{\omega_2})g(\ulflex{\omega_1}{\omega_2},\vecv_r''),
\end{align}
that is, we have
$$
g(\omega_1,\omega_2,\vecv_r'')
=(\omega_1)g(\omega_2,\vecv_r'')
+\pic(\lrflex{\omega_1}{\omega_2})g(\ulflex{\omega_1}{\omega_2},\vecv_r'').
$$
\end{prop}
\begin{proof}
By using the equation \eqref{eqn:definition of map g, by using D} and by using the partition \eqref{eqn:partition of D}, we calculate
\begin{align*}
g(\vecv_r)
=&\sum_{s\geq1}\sum_{(\omega_1a_1;b_1;\vecu)\in D_{2s}^{\geq2}(\vecv_r)}
	\pic(\lrflex{(\omega_1a_1)}{b_1})
	\ \bigl(\ulflex{(\omega_1a_1)}{b_1}\bigr)W(\vecu) \\
&+\sum_{s\geq1}\sum_{(\omega_1;b_1;\vecu)\in D_{2s}^1(\vecv_r)}
	\pic(\lrflex{\omega_1}{b_1})
	\ \bigl(\ulflex{\omega_1}{b_1}\bigr)W(\vecu).
\end{align*}
Because $a_1\neq\emptyset$ in the first summation, we have $\lrflex{(\omega_1a_1)}{b_1}=\lrflex{a_1}{b_1}$ and $\bigl(\ulflex{(\omega_1a_1)}{b_1}\bigr)=(\omega_1)(\ulflex{a_1}{b_1})$.
By using two bijections \eqref{eqn:bijection from Dgeq2 to D}, \eqref{eqn:bijection from D1 to D} for $t=2s$ and for $\omega=\vecv_r$, we get
\begin{align}\label{eqn: decomposition of g}
g(\vecv_r)
=&(\omega_1)\sum_{s\geq1}\sum_{(a_1;b_1;\vecu)\in D_{2s}(\vecv_r')}
	\pic(\lrflex{a_1}{b_1})
	\ \bigl(\ulflex{a_1}{b_1}\bigr) W(\vecu) \\
&+\sum_{s\geq1}\sum_{(b_1;\vecu)\in D_{2s-1}(\vecv_r')}
	\pic(\lrflex{\omega_1}{b_1})
	\ \bigl(\ulflex{\omega_1}{b_1}\bigr)W(\vecu). \nonumber
\end{align}
Here, by the equation \eqref{eqn:definition of map g, by using D}, the first term of the right hand side of the equation \eqref{eqn: decomposition of g} is equal to $(\omega_1)g(\vecv_r')$.
We calculate the second summation as below:
\begin{align*}
&\sum_{s\geq1}\sum_{(b_1;\vecu)\in D_{2s-1}(\vecv_r')}
	\pic(\lrflex{\omega_1}{b_1})
	\ \bigl(\ulflex{\omega_1}{b_1}\bigr)W(\vecu) \\
=&\pic(\lrflex{\omega_1}{\vecv_r'})\ \bigl(\ulflex{\omega_1}{\vecv_r'}\bigr)
+\sum_{s\geq2}\sum_{(\omega_2b_1;\vecu)\in D_{2s-1}^{\geq2}(\vecv_r')}
	\pic(\lrflex{\omega_1}{(\omega_2b_1)})
	\ \bigl(\ulflex{\omega_1}{(\omega_2b_1)}\bigr)W(\vecu) \\
&\hspace{4cm} +\sum_{s\geq2}\sum_{(\omega_2;\vecu)\in D_{2s-1}^1(\vecv_r')}
	\pic(\lrflex{\omega_1}{\omega_2})
	\ \bigl(\ulflex{\omega_1}{\omega_2}\bigr)W(\vecu).
\intertext{Here, for the first term, we have $\pic(\lrflex{\omega_1}{\vecv_r'})=\pic(\lrflex{\omega_1}{\omega_2})\pic(\lrflex{\omega_1}{\vecv_r''})=\pic(\lrflex{\omega_1}{\omega_2})\pic(\lrflex{\ulflex{\omega_1}{\omega_2}}{\vecv_r''})$ and $\bigl(\ulflex{\omega_1}{\vecv_r'}\bigr)=\bigl(\ulflex{(\ulflex{\omega_1}{\omega_2})}{\vecv_r''}\bigr)$.
For the second term, we have $\pic(\lrflex{\omega_1}{(\omega_2b_1)})=\pic(\lrflex{\omega_1}{\omega_2})\pic(\lrflex{\omega_1}{b_1})=\pic(\lrflex{\omega_1}{\omega_2})\pic(\lrflex{\ulflex{\omega_1}{\omega_2}}{b_1})$ and $\ulflex{\omega_1}{(\omega_2b_1)}=\ulflex{(\ulflex{\omega_1}{\omega_2})}{b_1}$.
So by applying two bijections \eqref{eqn:bijection from Dgeq2 to D}, \eqref{eqn:bijection from D1 to D} for $t=2s-1$ and for $\omega=\vecv_r'$ to the second and the third term, we get}
=&\pic(\lrflex{\omega_1}{\omega_2})\left\{\pic(\lrflex{\ulflex{\omega_1}{\omega_2}}{\vecv_r''})\ \bigl(\ulflex{(\ulflex{\omega_1}{\omega_2})}{\vecv_r''}\bigr)
\vphantom{\sum_{(\omega_2,a_2,b_2,\cdots ,a_s,b_s)\in D_{2s-1}^1(\vecv_r')}}\right. \\
&\left.+\sum_{s\geq2}\sum_{(b_1;\vecu)\in D_{2s-1}(\vecv_r'')}
	\pic(\lrflex{\ulflex{\omega_1}{\omega_2}}{b_1})
	\ \bigl(\ulflex{(\ulflex{\omega_1}{\omega_2})}{b_1}\bigr)W(\vecu)
+\sum_{s\geq2}\sum_{\vecu \in D_{2s-2}(\vecv_r'')}
	\ \bigl(\ulflex{\omega_1}{\omega_2}\bigr)W(\vecu)\right\}.
\intertext{By putting the first and the second term together and by changing the variable $s\geq2$ to $s\geq1$ in the third term, we have}
=&\pic(\lrflex{\omega_1}{\omega_2}) \\
&\hspace{-1.em}\cdot\left\{\sum_{s\geq1}\sum_{(b_1;\vecu)\in D_{2s-1}(\vecv_r'')}
	\pic(\lrflex{\ulflex{\omega_1}{\omega_2}}{b_1})
	\ \bigl(\ulflex{(\ulflex{\omega_1}{\omega_2})}{b_1}\bigr)W(\vecu)
+\sum_{s\geq1}\sum_{\vecu \in D_{2s}(\vecv_r'')}
	\ \bigl(\ulflex{\omega_1}{\omega_2}\bigr)W(\vecu)\right\}.
\intertext{By applying two bijections \eqref{eqn:bijection from Dgeq2 to D}, \eqref{eqn:bijection from D1 to D} for $t=2s$ and for $\omega=(\ulflex{\omega_1}{\omega_2},\vecv_r'')$ to the first and the second term, we have}
=&\pic(\lrflex{\omega_1}{\omega_2}) \\
&\hspace{-1.em}\cdot\left\{\sum_{s\geq1}\sum_{(\ulflex{\omega_1}{\omega_2};b_1;\vecu)\in D_{2s}^1((\ulflex{\omega_1}{\omega_2},\vecv_r''))}
	\pic(\lrflex{\ulflex{\omega_1}{\omega_2}}{b_1})
	\ \bigl(\ulflex{(\ulflex{\omega_1}{\omega_2})}{b_1}\bigr) W(\vecu) \right. \\
&\left.+\sum_{s\geq1}\sum_{(\ulflex{\omega_1}{\omega_2}a_1;b_1;\vecu)\in D_{2s}^{\geq2}((\ulflex{\omega_1}{\omega_2},\vecv_r''))}
	\pic(\lrflex{(\ulflex{\omega_1}{\omega_2}a_1)}{b_1})
	\ \bigl(\ulflex{(\ulflex{\omega_1}{\omega_2}a_1)}{b_1}\bigr) W(\vecu)\right\} \\
=&\pic(\lrflex{\omega_1}{\omega_2})
\sum_{s\geq1}\sum_{(a_1;b_1;\vecu)\in D_{2s}((\ulflex{\omega_1}{\omega_2},\vecv_r''))}
	\pic(\lrflex{a_1}{b_1})
	\ \bigl(\ulflex{a_1}{b_1}\bigr) W(\vecu) \\
=&\pic(\lrflex{\omega_1}{\omega_2})
\sum_{s\geq1}\sum_{\vecu\in D_{2s}((\ulflex{\omega_1}{\omega_2},\vecv_r''))}
	W(\vecu).
\end{align*}
By the equation \eqref{eqn:definition of map g, by using D}, we see that the last member is equal to $\pic(\lrflex{\omega_1}{\omega_2})g(\ulflex{\omega_1}{\omega_2},\vecv_r'')$, that is, the second term of the right hand side of the equation \eqref{eqn: decomposition of g} is equal to $\pic(\lrflex{\omega_1}{\omega_2})g(\ulflex{\omega_1}{\omega_2},\vecv_r'')$. 
Hence, we obtain the claim.
\end{proof}

\begin{lem}\label{lem:lemma on shuffle of map g}
Let $r,s\geq1$ and put $\alpha:=(\omega_2,\dots,\omega_r)$ and $\beta:=(\omega_{r+1},\dots,\omega_{r+s})$. Then we have
\begin{align}\label{eqn:lemma on shuffle of map g}
g(\omega_1,\alpha \shuffle_* \beta) 
=\sum_{p,q\geq0}\sum_{\substack{
	(b_0;\vecu)\in E_{2p+1}(\alpha) \\
	(d_0;\vecv)\in E_{2q+1}(\beta)}}
\pic(\lrflex{\omega_1}{(b_0d_0)})
\ \bigl(\ulflex{\omega_1}{(b_0d_0)}\bigr) \left\{ W(\vecu)\shuffle W(\vecv) \right\}.
\end{align}
\end{lem}
\begin{proof}
By the definition of $E_t(\omega)$, we have
\footnote{Here, it means that $E_t(\emptyset)$ is the set consisting of only the empty word $\emptyset$ when $t=1$ and is the emptyset when $t\geq2$.}
$$
E_{2p+1}(\emptyset)
=\left\{\begin{array}{cl}
	\{\emptyset\} & (p=0), \\
	\emptyset & (p\geq1),
\end{array}\right.
$$
for $p\geq0$, so the equation \eqref{eqn:lemma on shuffle of map g} for $\alpha=\emptyset$ (i.e. $r=1$) is equal to
\begin{align*}
g(\omega_1,\beta) 
=\sum_{q\geq0}\sum_{(d_0;\vecv)\in E_{2q+1}(\beta)}
\pic(\lrflex{\omega_1}{d_0})
\ \bigl(\ulflex{\omega_1}{d_0}\bigr) W(\vecv).
\end{align*}
Because this equation follows from Lemma \ref{lem:equivalent equation of g}, the equation \eqref{eqn:lemma on shuffle of map g} holds for $\alpha=\emptyset$.
Therefore, in the following, we assume $\alpha\neq\emptyset$, that is, $r\geq2$.

We prove \eqref{eqn:lemma on shuffle of map g} by induction on $r+s$ $(\geq3)$.
When $r+s=3$ (i.e. $r=2$, $s=1$), we have $\alpha=(\omega_2)$ and $\beta=(\omega_3)$ and 
$$
\alpha\shuffle_*\beta
=(\omega_2,\omega_3) + (\omega_3,\omega_2) 
- \pic(\lrflex{\omega_2}{\omega_3})\ (\ulflex{\omega_2}{\omega_3})
- \pic(\lrflex{\omega_3}{\omega_2})\ (\ulflex{\omega_3}{\omega_2}).
$$
So the left hand side of \eqref{eqn:lemma on shuffle of map g} is equal to
\begin{align*}
&g(\omega_1,\alpha\shuffle_*\beta) \\
&=g(\omega_1,\omega_2,\omega_3) + g(\omega_1,\omega_3,\omega_2) 
- \pic(\lrflex{\omega_2}{\omega_3})g(\omega_1,\ulflex{\omega_2}{\omega_3})
- \pic(\lrflex{\omega_3}{\omega_2})g(\omega_1,\ulflex{\omega_3}{\omega_2}) \\
&=(\omega_1,\omega_2,\omega_3)
	+\pic(\lrflex{\omega_2}{\omega_3})\ (\ulflex{\omega_1,\omega_2}{\omega_3})
	+\pic(\lrflex{\omega_1}{\omega_2})\ (\ulflex{\omega_1}{\omega_2},\omega_3)
	+\pic(\lrflex{\omega_1}{(\omega_2,\omega_3)})\ (\ulflex{\omega_1}{(\omega_2,\omega_3)}) \\
&\quad + (\omega_1,\omega_3,\omega_2)
	+\pic(\lrflex{\omega_3}{\omega_2})\ (\ulflex{\omega_1,\omega_3}{\omega_2})
	+\pic(\lrflex{\omega_1}{\omega_3})\ (\ulflex{\omega_1}{\omega_3},\omega_2)
	+\pic(\lrflex{\omega_1}{(\omega_3,\omega_2)})\ (\ulflex{\omega_1}{(\omega_3,\omega_2)}) \\
&\quad - \pic(\lrflex{\omega_2}{\omega_3})\left\{ (\omega_1,\ulflex{\omega_2}{\omega_3})+\pic(\lrflex{\omega_1}{\ulflex{\omega_2}{\omega_3}})\ (\ulflex{\omega_1}{\ulflex{\omega_2}{\omega_3}}) \right\} \\
&\qquad - \pic(\lrflex{\omega_3}{\omega_2})\left\{ (\omega_1,\ulflex{\omega_3}{\omega_2})+\pic(\lrflex{\omega_1}{\ulflex{\omega_3}{\omega_2}})\ (\ulflex{\omega_1}{\ulflex{\omega_3}{\omega_2}}) \right\}.
\intertext{By using $(\ulflex{\omega_1}{\ulflex{\omega_2}{\omega_3}})=(\ulflex{\omega_1}{\ulflex{\omega_3}{\omega_2}})=(\ulflex{\omega_1}{(\omega_2,\omega_3)})$ and by using Remark \ref{rem:shufflestar product-pic ver.}.(b), we get}
&=(\omega_1,\omega_2,\omega_3)
	+\pic(\lrflex{\omega_1}{\omega_2})\ (\ulflex{\omega_1}{\omega_2},\omega_3)
	+\pic(\lrflex{\omega_1}{(\omega_2,\omega_3)})\ (\ulflex{\omega_1}{(\omega_2,\omega_3)}) \\
&\quad + (\omega_1,\omega_3,\omega_2)
	+\pic(\lrflex{\omega_1}{\omega_3})\ (\ulflex{\omega_1}{\omega_3},\omega_2).
\end{align*}
While, by the definition of $E_t(\omega)$, we have
$$
E_{2p+1}(\omega_2)
=\left\{\begin{array}{cl}
	\{(\omega_2)\} & (p=0), \\
	\{(\emptyset;\omega_2;\emptyset)\} & (p=1), \\
	\emptyset & (p\geq2),
\end{array}\right.
\quad
E_{2q+1}(\omega_3)
=\left\{\begin{array}{cl}
	\{(\omega_3)\} & (q=0), \\
	\{(\emptyset;\omega_3;\emptyset)\} & (q=1), \\
	\emptyset & (q\geq2),
\end{array}\right.
$$
for $p,q\geq0$. Therefore, the right hand side of \eqref{eqn:lemma on shuffle of map g} is equal to
\begin{align*}
&\sum_{p,q\in \{0,1\}}\sum_{\substack{
	(b_0;\vecu)\in E_{2p+1}(\omega_2) \\
	(d_0;\vecv)\in E_{2q+1}(\omega_3)}}
\pic(\lrflex{\omega_1}{(b_0d_0)})
\ \bigl(\ulflex{\omega_1}{(b_0d_0)}\bigr) \left\{ W(\vecu)\shuffle W(\vecv) \right\} \\
&=\sum_{p\in \{0,1\}}\sum_{(b_0;\vecu)\in E_{2p+1}(\omega_2)}
\pic(\lrflex{\omega_1}{(b_0\omega_3)})
\ \bigl(\ulflex{\omega_1}{(b_0\omega_3)}\bigr) \left\{ W(\vecu)\shuffle \emptyset \right\} \\
&\quad +\sum_{p\in \{0,1\}}\sum_{(b_0;\vecu)\in E_{2p+1}(\omega_2)}
\pic(\lrflex{\omega_1}{b_0})
\ \bigl(\ulflex{\omega_1}{b_0}\bigr) \left\{ W(\vecu)\shuffle W(\omega_3;\emptyset) \right\} \\
&= \pic(\lrflex{\omega_1}{(\omega_2,\omega_3)})
\ \bigl(\ulflex{\omega_1}{(\omega_2,\omega_3)}\bigr) \left\{ \emptyset\shuffle \emptyset \right\} 
+\pic(\lrflex{\omega_1}{\omega_3})
\ \bigl(\ulflex{\omega_1}{\omega_3}\bigr) \left\{ W(\omega_2;\emptyset)\shuffle \emptyset \right\} \\
&\quad +\pic(\lrflex{\omega_1}{\omega_2})
\ \bigl(\ulflex{\omega_1}{\omega_2}\bigr) \left\{ \emptyset\shuffle W(\omega_3;\emptyset) \right\}
+\pic(\lrflex{\omega_1}{\emptyset})
\ \bigl(\ulflex{\omega_1}{\emptyset}\bigr) \left\{ W(\omega_2;\emptyset)\shuffle W(\omega_3;\emptyset) \right\} \\
&= \pic(\lrflex{\omega_1}{(\omega_2,\omega_3)})
\ \bigl(\ulflex{\omega_1}{(\omega_2,\omega_3)}\bigr)
+\pic(\lrflex{\omega_1}{\omega_3})
\ \bigl(\ulflex{\omega_1}{\omega_3},\omega_2\bigr) \\
&\quad +\pic(\lrflex{\omega_1}{\omega_2})
\ \bigl(\ulflex{\omega_1}{\omega_2},\omega_3\bigr)
+\left\{ \bigl(\omega_1,\omega_2,\omega_3\bigr)+\bigl(\omega_1,\omega_3,\omega_2\bigr) \right\}.
\end{align*}
Hence, the equation \eqref{eqn:lemma on shuffle of map g} holds for $r+s=3$.

Assume that \eqref{eqn:lemma on shuffle of map g} holds for $r+s\leq t\, (\geq3)$.
When $r+s=t+1$, by the equation \eqref{eqn:shufflestar product-pic ver.} of the product $\shuffle_*$, we get
\begin{align*}
&g(\omega_1,\alpha \shuffle_* \beta) \\
&=g(\omega_1,\omega_2,\alpha' \shuffle_* \beta)
+g(\omega_1,\omega_{r+1},\alpha \shuffle_* \beta') \\
&\quad -\pic(\lrflex{\omega_2}{\omega_{r+1}})g(\omega_1,\ulflex{\omega_2}{\omega_{r+1}},\alpha' \shuffle_* \beta')
-\pic(\lrflex{\omega_{r+1}}{\omega_2})g(\omega_1,\ulflex{\omega_{r+1}}{\omega_2},\alpha' \shuffle_* \beta').
\intertext{By applying Proposition \ref{prop:recurrence formula of map g} to each term, we get}
&=(\omega_1)g(\omega_2,\alpha' \shuffle_* \beta)
	+\pic(\lrflex{\omega_1}{\omega_2})g(\ulflex{\omega_1}{\omega_2},\alpha' \shuffle_* \beta) \\
&\quad +(\omega_1)g(\omega_{r+1},\alpha \shuffle_* \beta')
	+\pic(\lrflex{\omega_1}{\omega_{r+1}})g(\ulflex{\omega_1}{\omega_{r+1}},\alpha \shuffle_* \beta') \\
&\quad -\pic(\lrflex{\omega_2}{\omega_{r+1}})
	\left\{ (\omega_1)g(\ulflex{\omega_2}{\omega_{r+1}},\alpha' \shuffle_* \beta')
	+\pic(\lrflex{\omega_1}{\ulflex{\omega_2}{\omega_{r+1}}})g(\ulflex{\omega_1}{\ulflex{\omega_2}{\omega_{r+1}}},\alpha' \shuffle_* \beta') \right\} \\
&\quad -\pic(\lrflex{\omega_{r+1}}{\omega_2})
	\left\{ (\omega_1)g(\ulflex{\omega_{r+1}}{\omega_2},\alpha' \shuffle_* \beta')
	+\pic(\lrflex{\omega_1}{\ulflex{\omega_{r+1}}{\omega_2}})g(\ulflex{\omega_1}{\ulflex{\omega_{r+1}}{\omega_2}},\alpha' \shuffle_* \beta') \right\}.
\intertext{Because $\pic(\lrflex{\omega_1}{\ulflex{\omega_2}{\omega_{r+1}}})=\pic(\lrflex{\omega_1}{\omega_2})$ and $\ulflex{\omega_1}{\ulflex{\omega_2}{\omega_{r+1}}}=\ulflex{\omega_1}{(\omega_2,\omega_{r+1})}$, we have}
&=(\omega_1)g(\omega_2,\alpha' \shuffle_* \beta)
	+\pic(\lrflex{\omega_1}{\omega_2})g(\ulflex{\omega_1}{\omega_2},\alpha' \shuffle_* \beta) \\
&\quad +(\omega_1)g(\omega_{r+1},\alpha \shuffle_* \beta')
	+\pic(\lrflex{\omega_1}{\omega_{r+1}})g(\ulflex{\omega_1}{\omega_{r+1}},\alpha \shuffle_* \beta') \\
&\quad -\pic(\lrflex{\omega_2}{\omega_{r+1}})
	\left\{ (\omega_1)g(\ulflex{\omega_2}{\omega_{r+1}},\alpha' \shuffle_* \beta')
	+\pic(\lrflex{\omega_1}{\omega_2})g(\ulflex{\omega_1}{(\omega_2,\omega_{r+1})},\alpha' \shuffle_* \beta') \right\} \\
&\quad -\pic(\lrflex{\omega_{r+1}}{\omega_2})
	\left\{ (\omega_1)g(\ulflex{\omega_{r+1}}{\omega_2},\alpha' \shuffle_* \beta')
	+\pic(\lrflex{\omega_1}{\omega_{r+1}})g(\ulflex{\omega_1}{(\omega_2,\omega_{r+1})},\alpha' \shuffle_* \beta') \right\}.
\end{align*}
By using Remark \ref{rem:shufflestar product-pic ver.}.(b) (i.e. $\pic(\lrflex{\omega_2}{\omega_{r+1}})\pic(\lrflex{\omega_1}{\omega_2})+\pic(\lrflex{\omega_{r+1}}{\omega_2})\pic(\lrflex{\omega_1}{\omega_{r+1}})=\pic(\lrflex{\omega_1}{(\omega_2,\omega_{r+1})})$) and by rearranging each term, we get
\begin{align}\label{eqn:g(omega,shuffle_*)=3-term}
g(\omega_1,\alpha \shuffle_* \beta)
&=\left\{ (\omega_1)g(\omega_2,\alpha' \shuffle_* \beta)
	-\pic(\lrflex{\omega_2}{\omega_{r+1}})
	(\omega_1)g(\ulflex{\omega_2}{\omega_{r+1}},\alpha' \shuffle_* \beta') \right\} \\
&\quad + \left\{ (\omega_1)g(\omega_{r+1},\alpha \shuffle_* \beta')
	-\pic(\lrflex{\omega_{r+1}}{\omega_2})
	(\omega_1)g(\ulflex{\omega_{r+1}}{\omega_2},\alpha' \shuffle_* \beta') \right\} \nonumber\\
&\qquad +\left\{ \pic(\lrflex{\omega_1}{\omega_2})g(\ulflex{\omega_1}{\omega_2},\alpha' \shuffle_* \beta)
	+\pic(\lrflex{\omega_1}{\omega_{r+1}})g(\ulflex{\omega_1}{\omega_{r+1}},\alpha \shuffle_* \beta') \right. \nonumber\\
&\qquad\qquad \left. -\pic(\lrflex{\omega_1}{(\omega_2,\omega_{r+1})})
g(\ulflex{\omega_1}{(\omega_2,\omega_{r+1})},\alpha' \shuffle_* \beta') \right\}. \nonumber
\end{align}
In the following, we calculate the above each 3 terms.

Firstly, by induction hypothesis, we calculate the first term of the right hand side of \eqref{eqn:g(omega,shuffle_*)=3-term} as below:
\begin{align*}
&(\omega_1)g(\omega_2,\alpha' \shuffle_* \beta)
	-\pic(\lrflex{\omega_2}{\omega_{r+1}})
	(\omega_1)g(\ulflex{\omega_2}{\omega_{r+1}},\alpha' \shuffle_* \beta') \\
&=(\omega_1)
\sum_{p,q\geq0}\sum_{\substack{
		(b_0;\vecu)\in E_{2p+1}(\alpha') \\
		(d_0;\vecv)\in E_{2q+1}(\beta)}}
	\pic(\lrflex{\omega_2}{(b_0d_0)})
	\ \bigl(\ulflex{\omega_2}{(b_0d_0)}\bigr) \left\{ W(\vecu)\shuffle W(\vecv) \right\} \\
&\quad -\pic(\lrflex{\omega_2}{\omega_{r+1}})
	(\omega_1)
	\sum_{p,q\geq0}\sum_{\substack{
		(b_0;\vecu)\in E_{2p+1}(\alpha') \\
		(d_0;\vecv)\in E_{2q+1}(\beta')}}
	\pic(\lrflex{(\ulflex{\omega_2}{\omega_{r+1}})}{(b_0d_0)})
	\ \bigl(\ulflex{(\ulflex{\omega_2}{\omega_{r+1}})}{(b_0d_0)}\bigr) \left\{ W(\vecu)\shuffle W(\vecv) \right\}.
\end{align*}
Here, by using Remark \ref{rem:shufflestar product-pic ver.}.(b), we have
\begin{align}\label{eqn:pic relation-1}
\pic(\lrflex{\omega_2}{\omega_{r+1}})\pic(\lrflex{(\ulflex{\omega_2}{\omega_{r+1}})}{(b_0d_0)})
	&=\pic(\lrflex{\omega_2}{\omega_{r+1}})\pic(\lrflex{\omega_2}{(b_0d_0)}) \\
	&=\pic(\lrflex{\omega_2}{(b_0(\omega_{r+1}d_0))}), \nonumber\\ 
\label{eqn:pic relation-2}
\bigl(\ulflex{(\ulflex{\omega_2}{\omega_{r+1}})}{(b_0d_0)}\bigr)
	&=\bigl(\ulflex{\omega_2}{(b_0(\omega_{r+1}d_0))}\bigr),
\end{align}
so by using these equations and using the bijection \eqref{eqn:bijection from E to D}, we get
\begin{align*}
&(\omega_1)g(\omega_2,\alpha' \shuffle_* \beta)
	-\pic(\lrflex{\omega_2}{\omega_{r+1}})
	(\omega_1)g(\ulflex{\omega_2}{\omega_{r+1}},\alpha' \shuffle_* \beta') \\
&=(\omega_1)
\sum_{p,q\geq0}\sum_{\substack{
	(b_0;\vecu)\in E_{2p+1}(\alpha') \\
	(d_0;\vecv)\in E_{2q+1}(\beta)}}
\pic(\lrflex{\omega_2}{(b_0d_0)})
\ \bigl(\ulflex{\omega_2}{(b_0d_0)}\bigr) \left\{ W(\vecu)\shuffle W(\vecv) \right\} \\
&\quad -(\omega_1)
	\sum_{p,q\geq0}\sum_{\substack{
		(b_0;\vecu)\in E_{2p+1}(\alpha') \\
		(d_0;\vecv)\in D_{2q+1}(\beta)}}
	\pic(\lrflex{\omega_2}{(b_0d_0)})
	\ \bigl(\ulflex{\omega_2}{(b_0d_0)}\bigr) \left\{ W(\vecu)\shuffle W(\vecv) \right\}.
\intertext{By using the partition \eqref{eqn:partition of E} for $t=2q+1$ and by using $E_1(\beta)=D_1(\beta)$, we have}
&=(\omega_1)
\sum_{\substack{
	p\geq0 \\
	q\geq1}}\sum_{\substack{
	(b_0;\vecu)\in E_{2p+1}(\alpha') \\
	(d_0;\vecv)\in \{\emptyset\}\times D_{2q}(\beta)}}
\pic(\lrflex{\omega_2}{(b_0d_0)})
\ \bigl(\ulflex{\omega_2}{(b_0d_0)}\bigr) \left\{ W(\vecu)\shuffle W(\vecv) \right\} \\
&=(\omega_1)
\sum_{\substack{
	p\geq0 \\
	q\geq1}}\sum_{\substack{
	(b_0;\vecu)\in E_{2p+1}(\alpha') \\
	\vecv\in D_{2q}(\beta)}}
\pic(\lrflex{\omega_2}{b_0})
\ \bigl(\ulflex{\omega_2}{b_0}\bigr) \left\{ W(\vecu)\shuffle W(\vecv) \right\}.
\end{align*}
Hence, by using Lemma \ref{lem:equivalent equation of g} (change the variables $\vecv_r$ to $\beta$), we obtain
\begin{align}\label{eqn:first term formula of g}
&(\omega_1)g(\omega_2,\alpha' \shuffle_* \beta)
	-\pic(\lrflex{\omega_2}{\omega_{r+1}})
	(\omega_1)g(\ulflex{\omega_2}{\omega_{r+1}},\alpha' \shuffle_* \beta') \\
&=(\omega_1)
\sum_{p,q\geq0}\sum_{\substack{
	(b_0;\vecu)\in E_{2p+1}(\alpha') \\
	(d_0;\vecv)\in E_{2q+1}(\beta')}}
	\pic(\lrflex{\omega_2}{b_0})\pic(\lrflex{\omega_{r+1}}{d_0})
	\ \bigl(\ulflex{\omega_2}{b_0}\bigr) \left\{ W(\vecu)\shuffle \bigl(\ulflex{\omega_{r+1}}{d_0}\bigr)W(\vecv) \right\}. \nonumber
\end{align}
Similarly, for the second term of the right hand side of \eqref{eqn:g(omega,shuffle_*)=3-term}, we obtain
\begin{align}\label{eqn:second term formula of g}
&(\omega_1)g(\omega_{r+1},\alpha \shuffle_* \beta')
	-\pic(\lrflex{\omega_{r+1}}{\omega_2})
	(\omega_1)g(\ulflex{\omega_{r+1}}{\omega_2},\alpha' \shuffle_* \beta') \\
&=(\omega_1)
\sum_{p,q\geq0}\sum_{\substack{
	(b_0;\vecu)\in E_{2p+1}(\alpha') \\
	(d_0;\vecv)\in E_{2q+1}(\beta')}}
	\pic(\lrflex{\omega_2}{b_0})\pic(\lrflex{\omega_{r+1}}{d_0})
	\ \bigl(\ulflex{\omega_{r+1}}{d_0}\bigr) \left\{ \bigl(\ulflex{\omega_2}{b_0}\bigr)W(\vecu)\shuffle W(\vecv) \right\}. \nonumber
\end{align}
By taking the sum of \eqref{eqn:first term formula of g} and \eqref{eqn:second term formula of g} and by using the definition \eqref{eqn:shuffle product} of the product $\shuffle$, we get
\begin{align*}
&\left\{ (\omega_1)g(\omega_2,\alpha' \shuffle_* \beta)
	-\pic(\lrflex{\omega_2}{\omega_{r+1}})
	(\omega_1)g(\ulflex{\omega_2}{\omega_{r+1}},\alpha' \shuffle_* \beta') \right\} \\
&+\left\{ (\omega_1)g(\omega_{r+1},\alpha \shuffle_* \beta')
	-\pic(\lrflex{\omega_{r+1}}{\omega_2})
	(\omega_1)g(\ulflex{\omega_{r+1}}{\omega_2},\alpha' \shuffle_* \beta') \right\} \\
&=(\omega_1)
\sum_{p,q\geq0}\sum_{\substack{
	(b_0;\vecu)\in E_{2p+1}(\alpha') \\
	(d_0;\vecv)\in E_{2q+1}(\beta')}}
	\pic(\lrflex{\omega_2}{b_0})\pic(\lrflex{\omega_{r+1}}{d_0})
	\ \left\{ \bigl(\ulflex{\omega_2}{b_0}\bigr)W(\vecu) \shuffle \bigl(\ulflex{\omega_{r+1}}{d_0}\bigr)W(\vecv) \right\}.
\end{align*}
Hence, by using Lemma \ref{lem:equivalent equation of g}, we obtain
\begin{align}\label{eqn:first+second terms formula of g}
&\left\{ (\omega_1)g(\omega_2,\alpha' \shuffle_* \beta)
	-\pic(\lrflex{\omega_2}{\omega_{r+1}})
	(\omega_1)g(\ulflex{\omega_2}{\omega_{r+1}},\alpha' \shuffle_* \beta') \right\} \\
&+\left\{ (\omega_1)g(\omega_{r+1},\alpha \shuffle_* \beta')
	-\pic(\lrflex{\omega_{r+1}}{\omega_2})
	(\omega_1)g(\ulflex{\omega_{r+1}}{\omega_2},\alpha' \shuffle_* \beta') \right\} \nonumber\\
&=(\omega_1)
\sum_{p,q\geq1}\sum_{\substack{
	(b_0;\vecu)\in \{\emptyset\}\times D_{2p}(\alpha) \\
	(d_0;\vecv)\in \{\emptyset\}\times D_{2q}(\beta)}}
	\left\{ W(\vecu)\shuffle W(\vecv) \right\}. \nonumber
\end{align}

Secondly, by induction hypothesis, we calculate the third term of the right hand side of \eqref{eqn:g(omega,shuffle_*)=3-term} as below:
\begin{align*}
&\pic(\lrflex{\omega_1}{\omega_2})g(\ulflex{\omega_1}{\omega_2},\alpha' \shuffle_* \beta)
	+\pic(\lrflex{\omega_1}{\omega_{r+1}})g(\ulflex{\omega_1}{\omega_{r+1}},\alpha \shuffle_* \beta')  \\
& -\pic(\lrflex{\omega_1}{(\omega_2,\omega_{r+1})})
	g(\ulflex{\omega_1}{(\omega_2,\omega_{r+1})},\alpha' \shuffle_* \beta') \\
&=\pic(\lrflex{\omega_1}{\omega_2})
	\sum_{p,q\geq0}\sum_{\substack{
		(b_0;\vecu)\in E_{2p+1}(\alpha') \\
		(d_0;\vecv)\in E_{2q+1}(\beta)}}
	\pic(\lrflex{(\ulflex{\omega_1}{\omega_2})}{(b_0d_0)})
	\ \bigl(\ulflex{(\ulflex{\omega_1}{\omega_2})}{(b_0d_0)}\bigr) \left\{ W(\vecu)\shuffle W(\vecv) \right\} \\
&\quad +\pic(\lrflex{\omega_1}{\omega_{r+1}})
	\sum_{p,q\geq0}\sum_{\substack{
		(b_0;\vecu)\in E_{2p+1}(\alpha) \\
		(d_0;\vecv)\in E_{2q+1}(\beta')}}
	\pic(\lrflex{(\ulflex{\omega_1}{\omega_{r+1}})}{(b_0d_0)}) \\
&\hspace{6.5cm} \cdot \bigl(\ulflex{(\ulflex{\omega_1}{\omega_{r+1}})}{(b_0d_0)}\bigr) \left\{ W(\vecu)\shuffle W(\vecv) \right\} \\
&\quad -\pic(\lrflex{\omega_1}{(\omega_2,\omega_{r+1})})
	\sum_{p,q\geq0}\sum_{\substack{
		(b_0;\vecu)\in E_{2p+1}(\alpha') \\
		(d_0;\vecv)\in E_{2q+1}(\beta')}}
	\pic(\lrflex{(\ulflex{\omega_1}{(\omega_2,\omega_{r+1})})}{(b_0d_0)}) \\
&\hspace{6cm} \cdot \bigl(\ulflex{(\ulflex{\omega_1}{(\omega_2,\omega_{r+1})})}{(b_0d_0)}\bigr) \left\{ W(\vecu)\shuffle W(\vecv) \right\}.
\intertext{By using \eqref{eqn:pic relation-1} and \eqref{eqn:pic relation-2} and using the bijection \eqref{eqn:bijection from E to D}, we get}
&=\sum_{p,q\geq0}\sum_{\substack{
		(b_0;\vecu)\in D_{2p+1}(\alpha) \\
		(d_0;\vecv)\in E_{2q+1}(\beta)}}
	\pic(\lrflex{\omega_1}{(b_0d_0)})
	\ \bigl(\ulflex{\omega_1}{(b_0d_0)}\bigr) \left\{ W(\vecu)\shuffle W(\vecv) \right\} \\
&\quad +\sum_{p,q\geq0}\sum_{\substack{
		(b_0;\vecu)\in E_{2p+1}(\alpha) \\
		(d_0;\vecv)\in D_{2q+1}(\beta)}}
	\pic(\lrflex{\omega_1}{(b_0d_0)})
	\ \bigl(\ulflex{\omega_1}{(b_0d_0)}\bigr) \left\{ W(\vecu)\shuffle W(\vecv) \right\} \\
&\quad\quad -\sum_{p,q\geq0}\sum_{\substack{
		(b_0;\vecu)\in D_{2p+1}(\alpha) \\
		(d_0;\vecv)\in D_{2q+1}(\beta)}}
	\pic(\lrflex{\omega_1}{(b_0d_0)})
	\ \bigl(\ulflex{\omega_1}{(b_0d_0)}\bigr) \left\{ W(\vecu)\shuffle W(\vecv) \right\}.
\end{align*}
Hence, by using the partition \eqref{eqn:partition of E} and by using $E_1(\alpha)=D_1(\alpha)$, we obtain
\begin{align}\label{eqn:third term formula of g}
&\pic(\lrflex{\omega_1}{\omega_2})g(\ulflex{\omega_1}{\omega_2},\alpha' \shuffle_* \beta)
	+\pic(\lrflex{\omega_1}{\omega_{r+1}})g(\ulflex{\omega_1}{\omega_{r+1}},\alpha \shuffle_* \beta')  \\
& -\pic(\lrflex{\omega_1}{(\omega_2,\omega_{r+1})})
	g(\ulflex{\omega_1}{(\omega_2,\omega_{r+1})},\alpha' \shuffle_* \beta') \nonumber\\
&=\sum_{p,q\geq0}\sum_{\substack{
	(b_0;\vecu)\in D_{2p+1}(\alpha) \\
	(d_0;\vecv)\in E_{2q+1}(\beta)}}
\pic(\lrflex{\omega_1}{(b_0d_0)})
\ \bigl(\ulflex{\omega_1}{(b_0d_0)}\bigr) \left\{ W(\vecu)\shuffle W(\vecv) \right\} \nonumber\\
&\quad +\sum_{\substack{
	p\geq1 \\
	q\geq0}}\sum_{\substack{
	(b_0;\vecu)\in \{\emptyset\}\times D_{2p}(\alpha) \\
	(d_0;\vecv)\in D_{2q+1}(\beta)}}
\pic(\lrflex{\omega_1}{(b_0d_0)})
\ \bigl(\ulflex{\omega_1}{(b_0d_0)}\bigr) \left\{ W(\vecu)\shuffle W(\vecv) \right\}. \nonumber
\end{align}
Therefore, by taking the sum of \eqref{eqn:first+second terms formula of g} and \eqref{eqn:third term formula of g} and by using the partition \eqref{eqn:partition of E} and by using $E_1(\alpha)=D_1(\alpha)$, we have
\begin{align*}
g(\omega_1,\alpha \shuffle_* \beta) 
&=\sum_{p,q\geq0}\sum_{\substack{
		(b_0;\vecu)\in D_{2p+1}(\alpha) \\
		(d_0;\vecv)\in E_{2q+1}(\beta)}}
	\pic(\lrflex{\omega_1}{(b_0d_0)})
	\ \bigl(\ulflex{\omega_1}{(b_0d_0)}\bigr) \left\{ W(\vecu)\shuffle W(\vecv) \right\} \\
&\quad +\sum_{\substack{
	p\geq1 \\
	q\geq0}}\sum_{\substack{
	(b_0;\vecu)\in \{\emptyset\}\times D_{2p}(\alpha) \\
	(d_0;\vecv)\in E_{2q+1}(\beta)}}
\pic(\lrflex{\omega_1}{(b_0d_0)})
\ \bigl(\ulflex{\omega_1}{(b_0d_0)}\bigr) \left\{ W(\vecu)\shuffle W(\vecv) \right\} \\
&=\sum_{p,q\geq0}\sum_{\substack{
		(b_0;\vecu)\in E_{2p+1}(\alpha) \\
		(d_0;\vecv)\in E_{2q+1}(\beta)}}
	\pic(\lrflex{\omega_1}{(b_0d_0)})
	\ \bigl(\ulflex{\omega_1}{(b_0d_0)}\bigr) \left\{ W(\vecu)\shuffle W(\vecv) \right\}.
\end{align*}
Hence, we finish the proof.
\end{proof}

\begin{prop}\label{prop:shuffle of map g}
For $r,s\geq1$ and for $\alpha:=(\omega_1,\dots,\omega_r)$, $\beta:=(\omega_{r+1},\dots,\omega_{r+s})$, we have
\begin{equation}\label{eqn:shuffle of map g}
g(\alpha \shuffle_* \beta)=g(\alpha)\shuffle g(\beta).
\end{equation}
\end{prop}
\begin{proof}
By the definition \eqref{eqn:shufflestar product-pic ver.} of the product $\shuffle_*$, we have
\begin{align*}
g(\alpha \shuffle_* \beta)
&=g(\omega_1,\alpha' \shuffle_* \beta)
+g(\omega_{r+1},\alpha \shuffle_* \beta') \\
&\quad -\pic(\lrflex{\omega_1}{\omega_{r+1}})g(\ulflex{\omega_1}{\omega_{r+1}},\alpha' \shuffle_* \beta')
-\pic(\lrflex{\omega_{r+1}}{\omega_1})g(\ulflex{\omega_{r+1}}{\omega_1},\alpha' \shuffle_* \beta')
\end{align*}
By using Lemma \ref{lem:lemma on shuffle of map g}, we have
\begin{align*}
&g(\omega_1,\alpha' \shuffle_* \beta)
-\pic(\lrflex{\omega_1}{\omega_{r+1}})g(\ulflex{\omega_1}{\omega_{r+1}},\alpha' \shuffle_* \beta') \\
&=\sum_{p,q\geq0}\sum_{\substack{
	(b_0;\vecu)\in E_{2p+1}(\alpha') \\
	(d_0;\vecv)\in E_{2q+1}(\beta)}}
\pic(\lrflex{\omega_1}{(b_0d_0)})
\ \bigl(\ulflex{\omega_1}{(b_0d_0)}\bigr) \left\{ W(\vecu)\shuffle W(\vecv) \right\} \\
&\quad -\pic(\lrflex{\omega_1}{\omega_{r+1}})
\sum_{p,q\geq0}\sum_{\substack{
	(b_0;\vecu)\in E_{2p+1}(\alpha') \\
	(d_0;\vecv)\in E_{2q+1}(\beta')}}
\pic(\lrflex{(\ulflex{\omega_1}{\omega_{r+1}})}{(b_0d_0)})
\ \bigl(\ulflex{(\ulflex{\omega_1}{\omega_{r+1}})}{(b_0d_0)}\bigr) \left\{ W(\vecu)\shuffle W(\vecv) \right\} \\
&=\sum_{p,q\geq0}\sum_{\substack{
	(b_0;\vecu)\in E_{2p+1}(\alpha') \\
	(d_0;\vecv)\in E_{2q+1}(\beta)}}
\pic(\lrflex{\omega_1}{(b_0d_0)})
\ \bigl(\ulflex{\omega_1}{(b_0d_0)}\bigr) \left\{ W(\vecu)\shuffle W(\vecv) \right\} \\
&\quad -\sum_{p,q\geq0}\sum_{\substack{
	(b_0;\vecu)\in E_{2p+1}(\alpha') \\
	(d_0;\vecv)\in E_{2q+1}(\beta')}}
\pic(\lrflex{\omega_1}{(b_0(\omega_{r+1}d_0))})
\ \bigl(\ulflex{\omega_1}{(b_0(\omega_{r+1}d_0))}\bigr) \left\{ W(\vecu)\shuffle W(\vecv) \right\} \\
&=\sum_{p,q\geq0}\sum_{\substack{
	(b_0;\vecu)\in E_{2p+1}(\alpha') \\
	(d_0;\vecv)\in E_{2q+1}(\beta)}}
\pic(\lrflex{\omega_1}{(b_0d_0)})
\ \bigl(\ulflex{\omega_1}{(b_0d_0)}\bigr) \left\{ W(\vecu)\shuffle W(\vecv) \right\} \\
&\quad -\sum_{p,q\geq0}\sum_{\substack{
	(b_0;\vecu)\in E_{2p+1}(\alpha') \\
	(d_0;\vecv)\in D_{2q+1}(\beta)}}
\pic(\lrflex{\omega_1}{(b_0d_0)})
\ \bigl(\ulflex{\omega_1}{(b_0d_0)}\bigr) \left\{ W(\vecu)\shuffle W(\vecv) \right\}.
\intertext{By using the partition \eqref{eqn:partition of E} for $t=2q+1$ and by using $E_1(\beta)=D_1(\beta)$, we get}
&=\sum_{\substack{
	p\geq0 \\
	q\geq1}}\sum_{\substack{
	(b_0;\vecu)\in E_{2p+1}(\alpha') \\
	(d_0;\vecv)\in \{\emptyset\}\times D_{2q}(\beta)}}
\pic(\lrflex{\omega_1}{(b_0d_0)})
\ \bigl(\ulflex{\omega_1}{(b_0d_0)}\bigr) \left\{ W(\vecu)\shuffle W(\vecv) \right\} \\
&=\sum_{\substack{
	p\geq0 \\
	q\geq1}}\sum_{\substack{
	(b_0;\vecu)\in E_{2p+1}(\alpha') \\
	\vecv \in D_{2q}(\beta)}}
\pic(\lrflex{\omega_1}{b_0})
\ \bigl(\ulflex{\omega_1}{b_0}\bigr) \left\{ W(\vecu)\shuffle W(\vecv) \right\}.
\end{align*}
By using Lemma \ref{lem:equivalent equation of g} (change the variables $\vecv_r$ to $\beta$), we obtain
\begin{align}\label{eqn:first formula of g}
&g(\omega_1,\alpha' \shuffle_* \beta)
-\pic(\lrflex{\omega_1}{\omega_{r+1}})g(\ulflex{\omega_1}{\omega_{r+1}},\alpha' \shuffle_* \beta') \\
&=\sum_{p,q\geq0}\sum_{\substack{
	(b_0;\vecu)\in E_{2p+1}(\alpha') \\
	(d_0;\vecv) \in E_{2q+1}(\beta')}}
\pic(\lrflex{\omega_1}{b_0})\pic(\lrflex{\omega_{r+1}}{d_0})
\ \bigl(\ulflex{\omega_1}{b_0}\bigr) \left\{ W(\vecu)\shuffle \bigl(\ulflex{\omega_{r+1}}{d_0}\bigr)W(\vecv) \right\}. \nonumber
\end{align}
Similarly, we obtain
\begin{align}\label{eqn:second formula of g}
&g(\omega_{r+1},\alpha \shuffle_* \beta')
-\pic(\lrflex{\omega_{r+1}}{\omega_1})g(\ulflex{\omega_{r+1}}{\omega_1},\alpha' \shuffle_* \beta') \\
&=\sum_{p,q\geq0}\sum_{\substack{
	(b_0;\vecu)\in E_{2p+1}(\alpha') \\
	(d_0;\vecv) \in E_{2q+1}(\beta')}}
\pic(\lrflex{\omega_1}{b_0})\pic(\lrflex{\omega_{r+1}}{d_0})
\ \bigl(\ulflex{\omega_{r+1}}{d_0}\bigr) \left\{ \bigl(\ulflex{\omega_1}{b_0}\bigr)W(\vecu)\shuffle W(\vecv) \right\}. \nonumber
\end{align}
Therefore, by taking the sum of \eqref{eqn:first formula of g} and \eqref{eqn:second formula of g} and by using the definition \eqref{eqn:shuffle product} of the product $\shuffle$, we get
\begin{align*}
&g(\alpha \shuffle_* \beta) \\
&=\sum_{p,q\geq0}\sum_{\substack{
	(b_0;\vecu)\in E_{2p+1}(\alpha') \\
	(d_0;\vecv) \in E_{2q+1}(\beta')}}
\pic(\lrflex{\omega_1}{b_0})\pic(\lrflex{\omega_{r+1}}{d_0})
\ \left\{ \bigl(\ulflex{\omega_1}{b_0}\bigr)W(\vecu)\shuffle \bigl(\ulflex{\omega_{r+1}}{d_0}\bigr)W(\vecv) \right\} \\
&=\left\{ \sum_{p\geq0}\sum_{(b_0;\vecu)\in E_{2p+1}(\alpha')}
\pic(\lrflex{\omega_1}{b_0})\ \bigl(\ulflex{\omega_1}{b_0}\bigr)W(\vecu) \right\} \\
&\hspace{4.5cm} \shuffle 
\left\{ \sum_{q\geq0}\sum_{(d_0;\vecv) \in E_{2q+1}(\beta')}
\pic(\lrflex{\omega_{r+1}}{d_0})\ \bigl(\ulflex{\omega_{r+1}}{d_0}\bigr)W(\vecv) \right\}.
\intertext{By using Lemma \ref{lem:equivalent equation of g}, we obtain}
&=g(\alpha) \shuffle g(\beta).
\end{align*}
Hence, we obtain the claim.
\end{proof}

By the above theorem, we immediately obtain the following corollary.
\begin{cor}\label{cor:sh*(ganit(pic))(M)=sigma sh(M)}
For $M\in \overline{\BIMU}(\Gamma)$ and for $r,s\geq1$ and for $\alpha:=(\omega_1,\dots,\omega_r)$, $\beta:=(\omega_{r+1},\dots,\omega_{r+s})$, we have
\begin{align}\label{eqn:sh*(ganit(pic))(M)=sigma sh(M)}
&\shmap_* \bigl( \ganit_v(\pic)(M) \bigr)(\alpha;\beta) \\
&=\sum_{p,q\geq1}\sum_{\substack{
	(a_1;b_1;\cdots;a_p;b_p)\in D_{2p}(\alpha) \\
	(c_1;d_1;\cdots;c_q;d_q)\in D_{2q}(\beta)}}
\pic(\lrflex{a_1}{b_1})\cdots\pic(\lrflex{a_p}{b_p})\pic(\lrflex{c_1}{d_1})\cdots\pic(\lrflex{c_q}{d_q}) \nonumber\\
&\hspace{5cm} \cdot \shmap(M)\bigl(\ulflex{a_1}{b_1}\cdots \ulflex{a_p}{b_p};\ulflex{c_1}{d_1}\cdots \ulflex{c_q}{d_q}\bigr). \nonumber
\end{align}
\end{cor}

\begin{thm}[{cf. \cite[Proposition 6.2]{SS}, \cite[Lemma 4.4.2]{S-ARIGARI}}]\label{thm: ganit(B)(ARIal) subset ARIil}
The following two hold:
\begin{description}
	\item[(i)] $\ganit_v(\pic)(\overline{\ARI}(\Gamma)_\al) \subset \overline{\ARI}(\Gamma)_\il$.
	\item[(ii)] $\ganit_v(\pic)(\overline{\GARI}(\Gamma)_\as) \subset \overline{\GARI}(\Gamma)_\is$.
\end{description}
\end{thm}
\begin{proof}
Let $r,s\geq1$ and $\alpha:=(\omega_1,\dots,\omega_r)$, $\beta:=(\omega_{r+1},\dots,\omega_{r+s})$.
Note that we have
$$
\ulflex{a_1}{b_1}\cdots \ulflex{a_p}{b_p} \neq \emptyset
\quad (\mbox{resp.\quad } \ulflex{c_1}{d_1}\cdots \ulflex{c_q}{d_q} \neq \emptyset)
$$
for $(a_1;b_1;\cdots;a_p;b_p)\in D_{2p}(\alpha)$ (resp. $(c_1;d_1;\cdots;c_q;d_q)\in D_{2q}(\beta)$). \\
(i). Let $A\in\overline{\ARI}(\Gamma)_\al$.
By the definition of the alternal mould, we have $\shmap(A)(\omega;\eta)=0$ for $\omega,\eta\in V_\Z^\bullet\setminus\{\emptyset\}$.
Because by Corollary \ref{cor:sh*(ganit(pic))(M)=sigma sh(M)} we have
$$
\shmap_* \bigl( \ganit_v(\pic)(A) \bigr)(\alpha;\beta)=0,
$$
we get
$$
\shmap_* \bigl( \ganit_v(\pic)(A) \bigr)
= \ganit_v(\pic)(A) \otimes I + I \otimes \ganit_v(\pic)(A),
$$
that is, by Remark \ref{rem:gp-like, Lie-like of il, is}.(i), we see that $\ganit_v(\pic)(A)$ is alternil.
Hence, we obtain $\ganit_v(\pic)(\overline{\ARI}(\Gamma)_\al) \subset \overline{\ARI}(\Gamma)_\il$. \\
(ii). Let $S\in\overline{\GARI}(\Gamma)_\as$.
By Proposition \ref{prop:gp-like, Lie-like}.(ii), we have $\shmap(S)=S\otimes S$.
So by using Corollary \ref{cor:sh*(ganit(pic))(M)=sigma sh(M)}, we get
\begin{align*}
&\shmap_* \bigl( \ganit_v(\pic)(S) \bigr)(\alpha;\beta) \\
&=\sum_{p,q\geq1}\sum_{\substack{
	(a_1;b_1;\cdots;a_p;b_p)\in D_{2p}(\alpha) \\
	(c_1;d_1;\cdots;c_q;d_q)\in D_{2q}(\beta)}}
\pic(\lrflex{a_1}{b_1})\cdots\pic(\lrflex{a_p}{b_p})\pic(\lrflex{c_1}{d_1})\cdots\pic(\lrflex{c_q}{d_q}) \\
&\hspace{5cm} \cdot S\bigl(\ulflex{a_1}{b_1}\cdots \ulflex{a_p}{b_p}\bigr) S\bigl(\ulflex{c_1}{d_1}\cdots \ulflex{c_q}{d_q}\bigr) \\
&=\left\{
\sum_{p\geq1}\sum_{(a_1;b_1;\cdots;a_p;b_p)\in D_{2p}(\alpha)}
\pic(\lrflex{a_1}{b_1})\cdots\pic(\lrflex{a_p}{b_p})
S\bigl(\ulflex{a_1}{b_1}\cdots \ulflex{a_p}{b_p}\bigr)
\right\} \\
&\qquad \cdot \left\{
\sum_{q\geq1}\sum_{(c_1;d_1;\cdots;c_q;d_q)\in D_{2q}(\beta)}
\pic(\lrflex{c_1}{d_1})\cdots\pic(\lrflex{c_q}{d_q})
S\bigl(\ulflex{c_1}{d_1}\cdots \ulflex{c_q}{d_q}\bigr)
\right\} \\
&= \bigl( \ganit_v(\pic)(S) \bigr)(\alpha) \bigl( \ganit_v(\pic)(S) \bigr)(\beta) \\
&= \bigl( \ganit_v(\pic)(S) \otimes \ganit_v(\pic)(S) \bigr)(\alpha;\beta).
\end{align*}
Hence, we obtain $\shmap_* \bigl( \ganit_v(\pic)(S) \bigr)=\ganit_v(\pic)(S) \otimes \ganit_v(\pic)(S)$, that is, we see that $\ganit_v(\pic)(S)$ is symmetril.
Hence, we obtain $\ganit_v(\pic)(\overline{\GARI}(\Gamma)_\as)\in\overline{\GARI}(\Gamma)_\is$.
\end{proof}

\begin{remark}\label{rem: inverse map of ganit(pic)}
As an analogue of Proposition \ref{prop:recurrence formula of map g}, we have
\begin{align}\label{eqn:recurrence formula of map gpoc}
g_{\poc}(\vecv_r)
&=g_{\poc}(\vecv_{r-1})(\omega_r)
+\poc(\lrflex{\omega_{r-1}}{\omega_r})g_{\poc}(\vecv_{r-2},\ulflex{\omega_{r-1}}{\omega_r}),
\end{align}
for $r\geq2$. Here, $\poc$ is defined in \eqref{eqn:def of poc}.
By using this equation \eqref{eqn:recurrence formula of map gpoc}, we get an analogue of Proposition \ref{prop:shuffle of map g}, that is, we have
\begin{equation}\label{eqn:shuffle* of map gpoc}
g_{\poc}(\alpha \shuffle \beta)=g_{\poc}(\alpha)\shuffle_* g_{\poc}(\beta)
\end{equation}
for $r,s\geq1$ and for $\alpha:=(\omega_1,\dots,\omega_r)$, $\beta:=(\omega_{r+1},\dots,\omega_{r+s})$.
By using this equation, we obtain
\begin{align}\label{eqn:sh(ganit(poc))(M)=sigma sh*(M)}
&\shmap \bigl( \ganit_v(\poc)(M) \bigr)(\alpha;\beta) \\
&=\sum_{p,q\geq1}\sum_{\substack{
	(a_1;b_1;\cdots;a_p;b_p)\in D_{2p}(\alpha) \\
	(c_1;d_1;\cdots;c_q;d_q)\in D_{2q}(\beta)}}
\poc(\lrflex{a_1}{b_1})\cdots\poc(\lrflex{a_p}{b_p})\poc(\lrflex{c_1}{d_1})\cdots\poc(\lrflex{c_q}{d_q}) \nonumber\\
&\hspace{5cm} \cdot \shmap_*(M)\bigl(\ulflex{a_1}{b_1}\cdots \ulflex{a_p}{b_p};\ulflex{c_1}{d_1}\cdots \ulflex{c_q}{d_q}\bigr), \nonumber
\end{align}
as an analogue of Corollary \ref{cor:sh*(ganit(pic))(M)=sigma sh(M)}.
It should be noted that the equation \eqref{eqn:sh(ganit(poc))(M)=sigma sh*(M)} is equivalent to
\begin{align*}
\sum_{w\in sh(u,v)}B(w)
=\sum_{I,J}
\frac{1}{\prod_{i\in I}(v_i-v_{i+1})\prod_{j\in J}(v_j-v_{j+1})}
A_{|I'|,|J'|}(v_{I'},v_{J'})
\end{align*}
denoted in \cite[the end of page 55]{S-ARIGARI}.
Hence, similarly to the proof of Theorem \ref{thm: ganit(B)(ARIal) subset ARIil}, we obtain
\begin{align}
\ganit_v(\poc) \left( \overline{\ARI}(\Gamma)_\il \right) &\subset \overline{\ARI}(\Gamma)_\al, \\
\ganit_v(\poc) \left( \overline{\GARI}(\Gamma)_\is \right) &\subset \overline{\GARI}(\Gamma)_\as.
\end{align}
\end{remark}

\begin{thm}[{cf. \cite[Proposition 6.2]{SS}, \cite[Lemma 4.4.2]{S-ARIGARI}}]\label{thm: ganit(B) is automorphism between al and il or as and is}
The map $\ganit_v(\pic)$ induces a group isomorphism from $(\overline{\GARI}(\Gamma)_\as,\times)$ to $(\overline{\GARI}(\Gamma)_\is,\times)$ and induces a Lie algebra isomorphism from $(\overline{\ARI}(\Gamma)_\al,[,])$ to $(\overline{\ARI}(\Gamma)_\il,[,])$.
\end{thm}
\begin{proof}
By combining Theorem \ref{thm: ganit(B)(ARIal) subset ARIil} and Remark \ref{rem: inverse map of ganit(pic)}, we get
\begin{align*}
&\ganit_v(\pic)(\overline{\ARI}(\Gamma)_\al) = \overline{\ARI}(\Gamma)_\il, \\
&\ganit_v(\pic)(\overline{\GARI}(\Gamma)_\as) = \overline{\GARI}(\Gamma)_\is.
\end{align*}
Hence, by using Theorem \ref{thm: ganit(B) is automorphism}, we obtain the claim.
\end{proof}


\begin{cor}[{\cite[\S 4.7]{E-flex}}]\label{cor:commutative diagram of al,il,as,is}
The following diagram commutes:
\footnote{This diagram is not presented in \cite{SS} and \cite{S-ARIGARI}.}
$$
\xymatrix{
\overline{\GARI}(\Gamma)_\as \ar[rr]^{\ganit_v(\pic)} & & \overline{\GARI}(\Gamma)_\is  \\
\overline{\ARI}(\Gamma)_\al \ar[rr]_{\ganit_v(\pic)} \ar@{->}[u]^{\exp_\times}& & \overline{\ARI}(\Gamma)_\il \ar@{->}[u]_{\exp_\times}
}
$$
\end{cor}
\begin{proof}
By using Theorem \ref{thm:exp(ARIal)=GARIas etc}, Corollary \ref{cor:commutative diagram} and Theorem \ref{thm: ganit(B) is automorphism between al and il or as and is}, we obtain this claim.
\end{proof}

Theorem \ref{thm: ganit(B) is automorphism between al and il or as and is} is also used to prove the following.
\begin{thm}[{cf. \cite[Theorem 7.2]{SS}, \cite[Theorem 4.6.1]{S-ARIGARI}}]\label{thm:adari(pal) induce bijection}
The Lie algebra automorphism $\adari(\pal)$ on $\ARI(\Gamma)$ induces a bijection $\adari(\pal):\ARI(\Gamma)_{\underline{\al}*\underline{\al}} \longrightarrow \ARI(\Gamma)_{\underline{\al}*\underline{\il}}$.
\end{thm}

\begin{cor}
The vector space $\ARI(\Gamma)_{\underline{\al}*\underline{\il}}$ forms a Lie subalgebra of $\ARI(\Gamma)$ under the $\ari_u$-bracket.
\end{cor}

\begin{remark}
To finish the proof of Theorem \ref{thm:adari(pal) induce bijection}, we actually need the proof of another claim, that is, Theorem \ref{thm:expari(ARIal) = GARIas}.
We give this proof in Appendix \ref{sec:expari(ARIal)=GARIas}.
\end{remark}

\bigskip
\thanks{ {\it Acknowledgements}. The author is cordially grateful to Professor H. Furusho for guiding him towards this topic and for giving  useful suggestions to him.}

\appendix

\section{On $\expari(\ARI(\Gamma)_\al) =\GARI(\Gamma)_\as$}\label{sec:expari(ARIal)=GARIas}
In this appendix, we give the proof of Theorem \ref{thm:expari(ARIal) = GARIas}, whose claim is that
$\expari(\ARI(\Gamma)_\al) = \GARI(\Gamma)_\as$. 
In \S \ref{subsec:appendix, preparation}, we recall a claim in \cite{FK}, and we introduce a certain sequence and show the important property. These are useful to prove the above theorem.
In \S \ref{subsec:proof of expari(ARIal) =GARIas}, we actually give a proof of the above theorem.
\subsection{Preparation}\label{subsec:appendix, preparation}
The following claim in \cite{FK} is used to prove Theorem \ref{thm:expari(ARIal) = GARIas}.

\begin{lem}[{\cite[(A.3)]{FK} (cf. \cite[Proposition 1.13]{FK})}]\label{lem:alternality of arit}
For $A,B\in\ARI(\Gamma)_\al$, we have
$$
\arit(B)(A)\in\ARI(\Gamma)_\al.
$$
Here, $\arit(B)$ is a derivation on $\BIMU(\Gamma)$ (See \cite[Definition 1.9]{FK} for detailed definition of $\arit$).
\end{lem}

In \S \ref{sec:Mould}, we introduced the non-commutative free monoid $U_\Z^\bullet$ (Notation \ref{not:on expression of elements in V_Z}) generated by all element of the set $U_\Z$ with the empty word $\emptyset$ as the unit, and introduced the non-commutative polynomial $\Q$-algebra $\mathcal A_U:=\Q \langle U_\Z \rangle$ (\S \ref{subsec:al,il,as,is}) generated by $U_\Z$.
As an analogue of these, we consider the non-commutative free monoid $\N^\bullet$ generated by all natural number with the empty word $\emptyset$ as the unit, that is, $\N^\bullet=\bigsqcup_{k\geq0}\N^k$, and consider the non-commutative polynomial $\Q$-algebra $\Q \langle \N \rangle$ generated by all natural number.
We equip $\Q \langle \N \rangle$ the product $\shuffle$ defined in \eqref{eqn:shuffle product}, and then the pair $(\Q \langle \N \rangle, \shuffle)$ forms a commutative, associative, unital $\Q$-algebra.
Similarly to \eqref{eqn:shuffle product, coefficient def}, we define the family $\{\Sh{\Bf{\footnotesize $m$}}{\Bf{\footnotesize $n$}}{\Bf{\footnotesize $k$}}\}_{\Bf{\footnotesize $m$},\Bf{\footnotesize $n$},\Bf{\footnotesize $k$}\in\N^\bullet}$ in $\Z$ by
$$
\Bf{$m$} \shuffle \Bf{$n$}
=\sum_{\Bf{\footnotesize $k$}\in\N^\bullet}
\Sh{\Bf{$m$}}{\Bf{$n$}}{\Bf{$k$}} \Bf{$k$},
$$
for $\Bf{$m$},\Bf{$n$}\in\N^\bullet$.

\begin{definition}
Let $f=\{f(n_1,\dots,n_r)\}_{r\in\N,n_i\in\N}$ be a family in $\C$.
We call $f$ \textit{symmetral} if $f$ satisfies
$$
\sum_{\Bf{\footnotesize $k$}\in\N^\bullet}
\Sh{\Bf{$m$}}{\Bf{$n$}}{\Bf{$k$}}f(\Bf{$k$})
=f(\Bf{$m$})f(\Bf{$n$}),
$$
for $\Bf{$m$},\Bf{$n$}\in\N^\bullet$.
\end{definition}

We define the family $Ex=\{Ex(n_1,\dots,n_r)\}_{r\in\N,n_i\in\N}$ in $\Q$ by
\footnote{This $Ex$ is denoted in \cite[(2.52)]{E-flex}.}
\begin{equation}\label{eqn:def of Ex}
Ex(n_1,\dots,n_r)
:=\frac{1}{(n_1-1)!\cdots(n_r-1)!}\frac{1}{(n_1+\cdots+n_r)\cdots (n_{r-1}+n_r)n_r}
\end{equation}
for $r\in\mathbb{N}$ and $n_1,\dots,n_r\in\N$.

\begin{lem}\label{lem:Ex is symmetral}
The family $Ex$ is symmetral.
\end{lem}
\begin{proof}
Let $i,j\geq1$ and $\Bf{$m$}=(m_1,\dots,m_i)\in\N^\bullet$ and $\Bf{$n$}=(m_{i+1},\dots,m_{i+j})\in\N^\bullet$.
By the definition, it is sufficient to prove
\begin{equation}\label{eqn:symmetrality of Ex}
\sum_{\Bf{\footnotesize $k$}\in\N^\bullet}
\Sh{\Bf{$m$}}{\Bf{$n$}}{\Bf{$k$}}Ex(\Bf{$k$})
=Ex(\Bf{$m$})Ex(\Bf{$n$}).
\end{equation}
We prove this by induction on $k=i+j\geq2$.
When $k=2$, i.e., $i=j=1$, the left hand side of \eqref{eqn:symmetrality of Ex} is equal to
\begin{align*}
&Ex(m_1,m_2) + Ex(m_2,m_1) \\
&=\frac{1}{(m_1-1)!(m_2-1)!}\frac{1}{(m_1+m_2)m_2}
+ \frac{1}{(m_1-1)!(m_2-1)!}\frac{1}{(m_2+m_1)m_1} \\
&=\frac{1}{(m_1-1)!(m_2-1)!}\frac{m_1+m_2}{(m_1+m_2)m_1m_2} \\
&=Ex(m_1)Ex(m_2).
\end{align*}
Hence, the equation \eqref{eqn:symmetrality of Ex} holds for $k=2$.
Assume that the equation \eqref{eqn:symmetrality of Ex} holds for $2\leq k\leq k_0(\geq2)$.
When $k=k_0+1$, by putting $\Bf{$m$}'=(m_2,\dots,m_i)$ and $\Bf{$n$}'=(n_2,\dots,n_j)$, we have
\begin{align*}
\sum_{\Bf{\footnotesize $k$}\in\N^\bullet}\Sh{\Bf{$m$}}{\Bf{$n$}}{\Bf{$k$}}Ex(\Bf{$k$}) 
=\sum_{\Bf{\footnotesize $k$}\in\N^\bullet}
\left\{\Sh{\Bf{$m$}'}{\Bf{$n$}}{\Bf{$k$}}
Ex(m_1,\Bf{$k$})+\Sh{\Bf{$m$}}{\Bf{$n$}'}{\Bf{$k$}}
Ex(n_1,\Bf{$k$})\right\}.
\end{align*}
Here, by the definition \eqref{eqn:def of Ex}, we have $Ex(\Bf{$m$})=\frac{1}{(m_1-1)!}\frac{1}{|\Bf{$m$}|}Ex(\Bf{$m$}')$ (where $|\Bf{$m$}|:=m_1+\cdots+m_i$), so we get
$$
Ex(m_1,\Bf{$k$})=\frac{1}{(m_1-1)!}\frac{1}{|\Bf{$m$}|+|\Bf{$n$}|}Ex(\Bf{$k$}),
\quad Ex(n_1,\Bf{$k$})=\frac{1}{(n_1-1)!}\frac{1}{|\Bf{$m$}|+|\Bf{$n$}|}Ex(\Bf{$k$}).
$$
Therefore, we calculate
\begin{align*}
&\sum_{\Bf{\footnotesize $k$}\in\N^\bullet}\Sh{\Bf{$m$}}{\Bf{$n$}}{\Bf{$k$}}Ex(\Bf{$k$}) \\
&=\frac{1}{|\Bf{$m$}|+|\Bf{$n$}|}\sum_{\Bf{\footnotesize $k$}\in\N^\bullet}
\left\{\Sh{\Bf{$m$}'}{\Bf{$n$}}{\Bf{$k$}}
\frac{1}{(m_1-1)!}Ex(\Bf{$k$})+\Sh{\Bf{$m$}}{\Bf{$n$}'}{\Bf{$k$}}
\frac{1}{(n_1-1)!}Ex(\Bf{$k$})\right\}.
\intertext{By induction hypothesis, we have}
&=\frac{1}{|\Bf{$m$}|+|\Bf{$n$}|}
\left\{	\frac{1}{(m_1-1)!}Ex(\Bf{$m$}')Ex(\Bf{$n$})+
\frac{1}{(n_1-1)!}Ex(\Bf{$m$})Ex(\Bf{$n$}')\right\} \\
&=\frac{|\Bf{$m$}|}{|\Bf{$m$}|+|\Bf{$n$}|}
Ex(\Bf{$m$})Ex(\Bf{$n$})+
\frac{|\Bf{$n$}|}{|\Bf{$m$}|+|\Bf{$n$}|}
Ex(\Bf{$m$})Ex(\Bf{$n$}) \\
&=Ex(\Bf{$m$})Ex(\Bf{$n$}).
\end{align*}
Hence, we finish the proof.
\end{proof}

For $k\geq0$ and $M\in\BIMU(\Gamma)$, we define
\begin{equation*}
\preari_k(M):=
\left\{\begin{array}{ll}
	I & (k=0), \\
	\preari(\preari_{k-1}(M),M) & (k\geq1).
\end{array}\right.
\end{equation*}
We define the map $\expari:\ARI(\Gamma) \rightarrow \GARI(\Gamma)$ by
\footnote{This map is defined in \cite[(2.50)]{E-flex}.}
\begin{equation}\label{eqn:def of expari}
\expari(M):=\sum_{k\geq0}\frac{1}{k!}\preari_k(M),
\end{equation}
for $M\in\ARI(\Gamma)$.

\subsection{Proof}\label{subsec:proof of expari(ARIal) =GARIas}
In this subsection, we first prove key formula (Proposition \ref{prop: expari expansion}) which is displayed in \cite[(2.51)]{E-flex}.
By using this proposition, we show main claim (Theorem \ref{thm:expari(ARIal) = GARIas}) in this appendix.

\begin{prop}[{\cite[(2.51)]{E-flex}}]\label{prop: expari expansion}
For any $A\in\ARI(\Gamma)$, we have
\begin{equation}\label{eqn:expari expansion}
\expari(A)
=I+\sum_{\substack{
	\Bf{\footnotesize $m$}=(m_p)\in\mathbb{N}^r \\
	r\geq1}}
Ex(\Bf{$m$})A_{m_1}\times\cdots\times A_{m_r},
\end{equation}
where $A_m:=\arit(A)^{m-1}(A)$ for $m\in\mathbb{N}$.
\end{prop}
To prove this proposition, we show the following two lemmas.
\begin{lem}\label{lem:expansion of preari k}
There exists a family $\left\{C(m_1,\dots,m_r)\right\}_{r\in\mathbb{N}, (m_j)\in\mathbb{N}^r}$ in $\Z$ independent of the mould $A\in\BIMU(\Gamma)$ such that
\begin{equation}\label{eqn:def of preari k}
\preari_k(A)=\sum_{\substack{
	\Bf{\footnotesize $m$}=(m_p)\in\mathbb{N}^r \\
	m_1+\dots+m_r=k \\
	r\geq1 }}
C(\Bf{$m$})A_{m_1}\times\cdots\times A_{m_r}.
\end{equation}
for $k\in\mathbb{N}$.
\end{lem}
\begin{proof}
We prove the existence of $C(m_1,\dots,m_r)$ by induction on $k=m_1+\cdots+m_r\geq1$.
When $k=1$, the left hand side of \eqref{eqn:def of preari k} is equal to $A$, and the right hand side of \eqref{eqn:def of preari k} is equal to $C(1)A_1=C(1)A$.
So by putting $C(1):= 1$, we get the equation \eqref{eqn:def of preari k}.
Assume the equation \eqref{eqn:def of preari k} holds for $k\leq k_0 (\in\N)$.
When $k=k_0+1$, we have
\begin{align*}
&\preari_{k_0+1}(A) \\
&=\preari(\preari_{k_0}(A),A).
\intertext{By induction hypothesis, we calculate}
&=\preari \left(
\sum_{\substack{
	\Bf{\footnotesize $m$}=(m_p)\in\mathbb{N}^r \\
	m_1+\dots+m_r=k_0 \\
	r\geq1 }}
C(\Bf{$m$})A_{m_1}\times\cdots\times A_{m_r}
,A
\right) \\
&=\arit(A) \left(
\sum_{\substack{
	\Bf{\footnotesize $m$}=(m_p)\in\mathbb{N}^r \\
	m_1+\dots+m_r=k_0 \\
	r\geq1 }}
C(\Bf{$m$})A_{m_1}\times\cdots\times A_{m_r}
\right)
+ \left(
\sum_{\substack{
	\Bf{\footnotesize $m$}=(m_p)\in\mathbb{N}^r \\
	m_1+\dots+m_r=k_0 \\
	r\geq1 }}
C(\Bf{$m$})A_{m_1}\times\cdots\times A_{m_r}
\right)
\times A.
\intertext{Because $\arit(A)$ is a derivation and $A=A_1$, we have}
&=\sum_{\substack{
	\Bf{\footnotesize $m$}=(m_p)\in\mathbb{N}^r \\
	m_1+\dots+m_r=k_0 \\
	r\geq1}}
C(\Bf{$m$})
\sum_{i=1}^r
A_{m_1}\times\cdots\times A_{m_{i-1}}\times A_{m_i+1}\times A_{m_{i+1}}\times\cdots\times A_{m_r} \\
&\quad+\sum_{\substack{
	\Bf{\footnotesize $m$}=(m_p)\in\mathbb{N}^r \\
	m_1+\dots+m_r=k_0 \\
	r\geq1}}
C(\Bf{$m$})A_{m_1}\times\cdots\times A_{m_r}\times A_1 \\
&=\sum_{\substack{
	\Bf{\footnotesize $m$}=(m_p)\in\mathbb{N}^r \\
	m_1+\dots+m_r=k_0+1 \\
	r\geq1}}
\sum_{i=1}^r C(m_1,\dots,m_{i-1},m_i-1,m_{i+1},\dots,m_r)
A_{m_1}\times\cdots\times A_{m_r} \\
&\quad+\sum_{\substack{
	\Bf{\footnotesize $m$}=(m_p)\in\mathbb{N}^{r-1} \\
	m_1+\dots+m_{r-1}=k_0 \\
	r\geq2}}
C(m_1,\dots,m_{r-1})A_{m_1}\times\cdots\times A_{m_{r-1}}\times A_1.
\end{align*}
Here, for the first summation, we regard $C(m_1,\dots,m_r)$ as $0$ when there exists $i\in\{1,\dots,r\}$ such that $m_i=0$.
On the other hand, we have
\begin{equation*}
\preari_{k_0+1}(A)
=\sum_{\substack{
	\Bf{\footnotesize $m$}=(m_p)\in\mathbb{N}^r \\
	m_1+\dots+m_r=k_0+1 \\
	r\geq1 }}
C(\Bf{$m$})A_{m_1}\times\cdots\times A_{m_r}.
\end{equation*}
Therefore, we get the following recurrence formulae
\begin{equation}\label{eqn:recurrence formula of C}
\left\{\begin{array}{ll}
	C(m_1)=C(m_1-1) & (r=1), \\
	C(\Bf{$m$})=\displaystyle\sum_{i=1}^rC(m_1,\dots,m_{i-1},m_i-1,m_{i+1},\dots,m_r)+\delta_{m_r,1}\cdot C(m_1,\dots,m_{r-1}) & (r\geq 2),
\end{array}\right.
\end{equation}
for $\Bf{$m$}\in\mathbb{N}^r$ with $k_0+1=m_1+\cdots+m_r\geq2$.
Here, $\delta_{m,1}$ is the Kronecker delta and $C(m_1,\dots,m_r):=0$ if there exists $i\in\{1,\dots,r\}$ such that $m_i=0$.
Hence, by using these recurrence formulae, we obtain a family $\left\{C(m_1,\dots,m_r)\right\}_{r\in\mathbb{N}, (m_j)\in\mathbb{N}^r}$ in $\Z$.
\end{proof}

\begin{lem}\label{lem:C=const. Ex}
For $r\geq1$ and $\Bf{$m$}=(m_1,\dots,m_r)\in\N^r$, we have
$$
C(\Bf{$m$})=(m_1+\dots+m_r)!\cdot Ex(\Bf{$m$}).
$$ 
\end{lem}
\begin{proof}
By the definition \eqref{eqn:def of Ex}, we have $Ex(1)=1=C(1)$. So it is sufficient to prove that $(m_1+\dots+m_r)!\cdot Ex(\Bf{$m$})$ satisfies the recurrence formulae \eqref{eqn:recurrence formula of C}.
When $r=1$, we have $Ex(m)=\frac{1}{m!}$, so we get
\begin{equation*}
	m!\cdot Ex(m)=m!\cdot\frac{1}{m!}=1=(m-1)!\cdot Ex(m-1).
\end{equation*}
We prove the case of $r\geq2$. Note that we have
\begin{align}\label{eqn:simple expression of Ex}
&Ex(m_1,\dots,m_{i-1},m_i-1,m_{i+1},\dots,m_r) \\
&=\left\{(m_i-1)\prod_{k=1}^r\frac{1}{(m_k-1)!}\right\}
	\cdot\prod_{j=1}^i\frac{m_j+\cdots+m_r}{m_j+\cdots+m_r-1}
	\cdot \prod_{j=1}^r\frac{1}{m_j+\cdots+m_r}, \nonumber
\end{align}
for $1\leq i\leq r$.
When $m_r\neq1$, by using the expression \eqref{eqn:simple expression of Ex}, we calculate
\begin{align*}
&\sum_{i=1}^r(m_1+\dots+m_r-1)!\cdot Ex(m_1,\dots,m_{i-1},m_i-1,m_{i+1},\dots,m_r) \\
&=(m_1+\dots+m_r-1)! \\
&\qquad \cdot \sum_{i=1}^r\left\{(m_i-1)\prod_{k=1}^r\frac{1}{(m_k-1)!}\right\}
	\cdot\prod_{j=1}^i\frac{m_j+\cdots+m_r}{m_j+\cdots+m_r-1}
	\cdot \prod_{j=1}^r\frac{1}{m_j+\cdots+m_r} \\
&=(m_1+\dots+m_r-1)! \\ 
&\qquad \cdot \prod_{k=1}^r\left\{\frac{1}{(m_k-1)!}\frac{1}{m_k+\cdots+m_r}\right\}
	\sum_{i=1}^r\left\{(m_i-1)\prod_{j=1}^i\frac{m_j+\cdots+m_r}{m_j+\cdots+m_r-1}\right\}.
\end{align*}
Here, we have
{\footnotesize
\begin{align*}
&m_1+\cdots+m_r \\
&=\frac{m_1+\cdots+m_r}{m_1+\cdots+m_r-1}\cdot\{(m_1-1)+(m_2+\cdots+m_r)\} \\
&=\frac{m_1+\cdots+m_r}{m_1+\cdots+m_r-1}\cdot\left\{(m_1-1)+\frac{m_2+\cdots+m_r}{m_2+\cdots+m_r-1}\cdot\{(m_2-1)+(m_3+\cdots+m_r)\}\right\}
\intertext{{\normalsize By using this transformation repeatedly, we get}}
&=\frac{m_1+\cdots+m_r}{m_1+\cdots+m_r-1} \\
&\hspace{.2cm}\cdot\left\{(m_1-1)+\frac{m_2+\cdots+m_r}{m_2+\cdots+m_r-1}\cdot\left\{(m_2-1)+\cdots +\frac{m_{r-1}+m_r}{m_{r-1}+m_r-1}\left\{(m_{r-1}-1)+\frac{m_r}{m_r-1}(m_r-1)\right\}\cdots\right\}\right\} \\
&=\sum_{k=1}^r\left\{(m_k-1)\prod_{j=1}^k\frac{m_j+\cdots+m_r}{m_j+\cdots+m_r-1}\right\}.
\end{align*}}
Therefore, we have
\begin{align*}
&\sum_{k=1}^r(m_1+\dots+m_r-1)!\cdot Ex(m_1,\dots,m_{k-1},m_k-1,m_{k+1},\dots,m_r) \\
&=(m_1+\dots+m_r-1)!\prod_{	i=1}^r\left\{\frac{1}{(m_i-1)!}\frac{1}{m_i+\cdots+m_r}\right\}
(m_1+\cdots+m_r) \\
&=(m_1+\dots+m_r)!\cdot Ex(m_1,\dots,m_r),
\end{align*}
that is, we obtain the recurrence formulae \eqref{eqn:recurrence formula of C} for $m_r\neq1$.
Similar to the case of $m_r\neq1$, we obtain the case of $m_r=1$ by using the following equation
\begin{equation*}
	m_1+\cdots+m_{r-1}+1=\sum_{k=1}^{r-1}\left\{(m_k-1)\prod_{j=1}^k\frac{m_j+\cdots+m_{r-1}+1}{m_j+\cdots+m_{r-1}}\right\}+\prod_{j=1}^{r-1}\frac{m_j+\cdots+m_{r-1}+1}{m_j+\cdots+m_{r-1}}.
\end{equation*}
Hence, we get the claim.
\end{proof}

By using the above two lemmas, we prove Proposition \ref{prop: expari expansion}:

\medskip
\noindent
\textit{Proof of Proposition \ref{prop: expari expansion}.}
By using Lemma \ref{lem:expansion of preari k} and Lemma \ref{lem:C=const. Ex}, we have
\begin{align*}
\expari(A)
&=I+\sum_{k\geq1}\frac{1}{k!}\preari_k(A)
\intertext{By using Lemma \ref{lem:expansion of preari k}, we get}
&=I+\sum_{k\geq1}\frac{1}{k!}\sum_{\substack{
	\Bf{\footnotesize $m$}=(m_p)\in\mathbb{N}^r \\
	m_1+\dots+m_r=k \\
	r\geq1 }}
C(\Bf{$m$})A_{m_1}\times\cdots\times A_{m_r}
\intertext{By using Lemma \ref{lem:C=const. Ex}, we calculate}
&=I+\sum_{k\geq1}\sum_{\substack{
	\Bf{\footnotesize $m$}=(m_p)\in\mathbb{N}^r \\
	m_1+\dots+m_r=k \\
	r\geq1 }}
Ex(\Bf{$m$})A_{m_1}\times\cdots\times A_{m_r} \\
&=I + \sum_{\substack{
	\Bf{\footnotesize $m$}=(m_p)\in\mathbb{N}^r \\
	r\geq1}}
Ex(\Bf{$m$})A_{m_1}\times\cdots\times A_{m_r}.
\end{align*}
So we obtain the claim. \hfill $\Box$\\

By using Proposition \ref{prop: expari expansion}, we prove the following theorem.
\begin{thm}[{\cite[Proposition 2.6.1]{S-ARIGARI}}]\label{thm:expari(ARIal) = GARIas}
We have
\begin{equation*}
	\expari(\ARI(\Gamma)_\al)=\GARI(\Gamma)_\as.
\end{equation*}
\end{thm}
\begin{proof}
We only prove $\expari(\ARI(\Gamma)_\al) \subset \GARI(\Gamma)_\as$. Let $A\in\ARI(\Gamma)_\al$, that is, assume that $A$ satisfies $\mathpzc{Sh}(A)=A\otimes I+I\otimes A$. Then we show $\expari(A)\in\GARI(\Gamma)_\as$, that is,
\begin{equation*}
\mathpzc{Sh}(\expari(A))=\expari(A) \otimes \expari(A).
\end{equation*}
By using Proposition \ref{prop: expari expansion}, we have
\begin{align*}
&\mathpzc{Sh}(\expari(A)) \\
&=\mathpzc{Sh}\left(
I+\sum_{\substack{
	\Bf{\footnotesize $m$}=(m_p)\in\mathbb{N}^r \\
	r\geq1}}
Ex(\Bf{$m$})A_{m_1}\times\cdots\times A_{m_r}
\right) \\
&=\mathpzc{Sh}(I)
+\sum_{\substack{
	\Bf{\footnotesize $m$}=(m_p)\in\mathbb{N}^r \\
	r\geq1}}
Ex(\Bf{$m$})\mathpzc{Sh}\left(A_{m_1}\times\cdots\times A_{m_r}\right) \\
&=\mathpzc{Sh}(I)
+\sum_{\substack{
	\Bf{\footnotesize $m$}=(m_p)\in\mathbb{N}^r \\
	r\geq1}}
Ex(\Bf{$m$})\mathpzc{Sh}\left(A_{m_1}\right)\times\cdots\times \mathpzc{Sh}\left(A_{m_r}\right).
\intertext{By Lemma \ref{lem:alternality of arit} and by the alternality of $A$, we have $A_m=\arit(A)^{m-1}(A)\in\ARI(\Gamma)_\al$, that is, we get $\mathpzc{Sh}(A_m)=A_m\otimes I+I\otimes A_m$. So we calculate}
&=\mathpzc{Sh}(I)
+\sum_{\substack{
	\Bf{\footnotesize $m$}=(m_p)\in\mathbb{N}^r \\
	r\geq1}}
Ex(\Bf{$m$})\left(A_{m_1}\otimes I+I\otimes A_{m_1}\right)\times\cdots\times \left(A_{m_r}\otimes I+I\otimes A_{m_r}\right) \\
&=\mathpzc{Sh}(I)
+\sum_{\substack{
	\Bf{\footnotesize $m$}=(m_p)\in\mathbb{N}^r \\
	r\geq1}}
Ex(\Bf{$m$})
\left\{\vphantom{\sum_{\substack{
	i+j=r \\
	i,j\geq1}}}
\left(A_{m_1}\times\cdots\times A_{m_r}\right)
\otimes I \right. 
+I\otimes\left(A_{m_1}\times\cdots\times A_{m_r}\right) \\
&\qquad+\left.\sum_{\substack{
	i+j=r \\
	i,j\geq1}}
\sum_{\sigma\in\sh(i,j)}
\left(A_{m_{\sigma^{-1}(1)}}\times\cdots\times A_{m_{\sigma^{-1}(i)}}\right)
\otimes\left(A_{m_{\sigma^{-1}(i+1)}}\times\cdots\times A_{m_{\sigma^{-1}(r)}}\right) \right\} \\
&=\mathpzc{Sh}(I)
+\left\{\sum_{\substack{
	\Bf{\footnotesize $m$}=(m_p)\in\mathbb{N}^i \\
	i\geq1}}
Ex(\Bf{$m$}) A_{m_1}\times\cdots\times A_{m_i}\right\}
\otimes I \\
&\quad + I\otimes\left\{\sum_{\substack{
	\Bf{\footnotesize $m$}=(m_p)\in\mathbb{N}^j \\
	j\geq1}}
Ex(\Bf{$m$}) A_{m_1}\times\cdots\times A_{m_j}\right\} 
+\sum_{\substack{
	\Bf{\footnotesize $m$}=(m_p)\in\mathbb{N}^{i+j} \\
	i,j\geq1}} \\
&\quad \cdot \left\{\sum_{\sigma\in\sh(i,j)}
Ex(\Bf{$m$})\left(A_{m_{\sigma^{-1}(1)}}\times\cdots\times A_{m_{\sigma^{-1}(i)}}\right)
\otimes\left(A_{m_{\sigma^{-1}(i+1)}}\times\cdots\times A_{m_{\sigma^{-1}(i+j)}}\right) \right\}.
\intertext{Because $\left\{m_{\sigma^{-1}(1)},\cdots,m_{\sigma^{-1}(i+j)}\right\}=\left\{m_1,\dots,m_{i+j}\right\}$ for any $\sigma\in\sh(i,j)$, we get}
&=\mathpzc{Sh}(I)
+\left\{\sum_{\substack{
	\Bf{\footnotesize $m$}=(m_p)\in\mathbb{N}^i \\
	i\geq1}}
Ex(\Bf{$m$}) A_{m_1}\times\cdots\times A_{m_i}\right\}
\otimes I \\
&\quad + I\otimes\left\{\sum_{\substack{
	\Bf{\footnotesize $m$}=(m_p)\in\mathbb{N}^j \\
	j\geq1}}
Ex(\Bf{$m$}) A_{m_1}\times\cdots\times A_{m_j}\right\} 
+\sum_{\substack{
	\Bf{\footnotesize $m$}=(m_p)\in\mathbb{N}^{i+j} \\
	i,j\geq1}} \\
&\quad \cdot\left\{\sum_{\sigma\in\sh(i,j)}
Ex\left(m_{\sigma^{-1}(1)},\cdots,m_{\sigma^{-1}(i+j)}\right)
\left(A_{m_1}\times\cdots\times A_{m_i}\right)
\otimes\left(A_{m_{i+1}}\times\cdots\times A_{m_{i+j}}\right) \right\}.
\end{align*}
Here, by Lemma \ref{lem:Ex is symmetral}, $Ex$ is symmetral, so we have
\begin{equation*}
\sum_{\sigma\in\sh(i,j)}
Ex\left(m_{\sigma^{-1}(1)},\cdots,m_{\sigma^{-1}(i+j)}\right)
=Ex(m_1,\dots,m_i)Ex(m_{i+1},\dots,m_{i+j}).
\end{equation*}
Therefore, we calculate
{\small
\begin{align*}
&\mathpzc{Sh}(\expari(A)) \\
&=\mathpzc{Sh}(I)
+\left\{\sum_{\substack{
	\Bf{\footnotesize $m$}=(m_p)\in\mathbb{N}^i \\
	i\geq1}}
Ex(\Bf{$m$}) A_{m_1}\times\cdots\times A_{m_i}\right\}
\otimes I \\
&\quad+I\otimes\left\{\sum_{\substack{
	\Bf{\footnotesize $m$}=(m_p)\in\mathbb{N}^j \\
	j\geq1}}
Ex(\Bf{$m$}) A_{m_1}\times\cdots\times A_{m_j}\right\} 
+\sum_{\substack{
	\Bf{\footnotesize $m$}=(m_p)\in\mathbb{N}^{i+j} \\
	i,j\geq1}} \\
&\quad\cdot\left\{Ex(m_1,\dots,m_i)Ex(m_{i+1},\dots,m_{i+j})
\left(A_{m_1}\times\cdots\times A_{m_i}\right)
\otimes\left(A_{m_{i+1}}\times\cdots\times A_{m_{i+j}}\right) \right\} \\
&=\mathpzc{Sh}(I)
+\left\{\sum_{\substack{
	\Bf{\footnotesize $m$}=(m_p)\in\mathbb{N}^i \\
	i\geq1}}
Ex(\Bf{$m$}) A_{m_1}\times\cdots\times A_{m_i}\right\}
\otimes I \\
&\quad+I\otimes\left\{\sum_{\substack{
	\Bf{\footnotesize $m$}=(m_p)\in\mathbb{N}^j \\
	j\geq1}}
Ex(\Bf{$m$}) A_{m_1}\times\cdots\times A_{m_j}\right\} \\
&\quad+\left\{\sum_{\substack{
	\Bf{\footnotesize $m$}=(m_p)\in\mathbb{N}^i \\
	i\geq1}}
Ex(\Bf{$m$})
A_{m_1}\times\cdots\times A_{m_i}\right\} 
\otimes\left\{\sum_{\substack{
	\Bf{\footnotesize $m$}=(m_p)\in\mathbb{N}^j \\
	j\geq1}}
Ex(\Bf{$m$})A_{m_1}\times\cdots\times A_{m_j} \right\}.
\intertext{\normalsize Because $\mathpzc{Sh}(I)=I\otimes I$, we get}
&=\left\{I+\sum_{\substack{
	\Bf{\footnotesize $m$}=(m_p)\in\mathbb{N}^i \\
	i\geq1}}
Ex(\Bf{$m$})
A_{m_1}\times\cdots\times A_{m_i}\right\}
\otimes\left\{I+\sum_{\substack{
	\Bf{\footnotesize $m$}=(m_p)\in\mathbb{N}^j \\
	j\geq1}}
Ex(\Bf{$m$})A_{m_1}\times\cdots\times A_{m_j} \right\} \\
&=\expari(A) \otimes \expari(A)
\end{align*}}
Hence, we obtain $\expari(A)\in\GARI(\Gamma)_\as$, that is, $\expari(\ARI(\Gamma)_\al) \subset \GARI(\Gamma)_\as$.
\end{proof}


\end{document}